\documentclass{article}
\usepackage{authblk}
\usepackage{bookmark}
\usepackage[margin = 4cm]{geometry} 
\usepackage{amsthm, amssymb, latexsym, graphics}
\usepackage{amsmath,amsfonts}
\usepackage{color}
\usepackage{hyperref}
\usepackage{todonotes}

\usepackage{caption}    
\usepackage[style = alphabetic, doi = false, url = false]{biblatex}
\newtheorem{theorem}{Theorem}[section]
\newtheorem{lemma}[theorem]{Lemma}
\newtheorem{definition}{Definition}[section]
\newtheorem{proposition}[theorem]{Proposition}
\newtheorem{corollary}[theorem]{Corollary}
\theoremstyle{definition}

\theoremstyle{definition}

\newtheorem{remark}{Remark}[section]

\DeclareMathOperator*{\essinf}{ess\,inf}
\DeclareMathOperator*{\esssup}{ess\,sup}

\newcommand\R{{\mathbb R}}

\newcommand{\gplus}{\gamma^+}
\newcommand{\gminus}{\gamma^-}
\newcommand{\Pra}{\rm{Pr}}
\newcommand{\Ra}{{\rm{Ra}}}
\newcommand{\Nu}{{\rm{Nu}}}
\newcommand{\la}{\langle}
\newcommand{\ra}{\rangle}

\providecommand{\keywords}[1]
{
  \noindent\small	
  \textbf{Keywords:} #1
}

\title{Bounds on buoyancy driven flows with Navier-slip conditions on rough boundaries}

\author[1]{Fabian Bleitner}
\author[2]{Camilla Nobili}

\affil[1]{\footnotesize{Department of Mathematics, University of Hamburg, Germany}}
\affil[2]{\footnotesize{School of Mathematics and Physics, University of Surrey, United Kingdom}}

\date{}

\begin{document}

\maketitle

\begin{abstract}
We consider two-dimensional Rayleigh-B\'enard convection with  Navier-slip and fixed temperature boundary conditions at the two horizontal rough walls described by the height function $h$.
We prove rigorous upper bounds on the Nusselt number $\Nu$ which capture the dependence on the curvature of the boundary $\kappa$ and the (non-constant) friction coefficient $\alpha$ explicitly.
If $h\in W^{2,\infty}$ and $\kappa$ satisfies a smallness condition with respect to $\alpha$, we find 
\begin{equation*}
 \Nu\lesssim \Ra^{\frac 12}+\|\kappa\|_{\infty}\,,
\end{equation*}
where $\Ra$ is the Rayleigh number, which agrees with the predicted Spiegel-Kraichnan scaling when $\kappa=0$.
This bound is obtained via local regularity estimates in a small strip at the boundary.
When $h\in W^{3,\infty}$, the functions $\kappa$ and $\alpha$ are sufficiently small in $L^{\infty}$ and the Prandtl number $\Pr$ is sufficiently large, we prove upper bounds using the background field method, which interpolate between $\Ra^{\frac 12}$ and $\Ra^{\frac{5}{12}}$ with non-trivial dependence on $\alpha$ and $\kappa$.
These bounds agree with the result in \cite{drivasNguyenNobiliBoundsOnHeatFluxForRayleighBenardConvectionBetweenNavierSlipFixedTemperatureBoundaries} for flat boundaries and constant friction coefficient.
Furthermore, in the regime $\Pr\geq \Ra^{\frac 57}$, we improve the $\Ra^{\frac 12}$-upper bound, showing
$$\Nu\lesssim_{\alpha,\kappa}\Ra^{\frac 37}\,,$$
where $\lesssim_{\alpha,\kappa}$ hides an additional dependency of the implicit constant on $\alpha$ and $\kappa$. 
\end{abstract}



\section{Introduction}
In this paper we deal with the Rayleigh-B\'enard convection problem, modelled by the Boussinesq system for the velocity field $u=(u_1,u_2)\in \R^2$, the scalar temperature field $T$ and pressure $p$
\begin{align}
    \frac{1}{\Pra}(u_t + u\cdot \nabla u ) -\Delta u +\nabla p &= \Ra T e_2\tag{NS}\label{navierStokes}\\
    \nabla \cdot u &= 0\tag{DF}\label{divergenceFreeIncompressibilityCondition}\\
    T_t +u\cdot \nabla T - \Delta T &= 0\tag{AD}\label{heatEquation}
\end{align}
set in the regular and bounded domain 
\begin{align*}
    \Omega=\{(y_1,y_2) \in \mathbb{R}^2 \ | \ 0 < y_1 < \Gamma, h(y_1) < y_2 < 1+h(y_1)\}
\end{align*}
where $h$ describes the height of the bottom boundary and $e_2=(0,1)$. We define the initial data by $u_0(y_1,y_2)=u(y_1,y_2,t)|_{t=0}$ and $T_0(y_1,y_2)=T(y_1,y_2,t)|_{t=0}$. The nondimensional numbers $\Pr$ and $\Ra$ are the Prandtl and Rayleigh number respectively. For a physical definition of these numbers see \eqref{Pr_def} and \eqref{rel} in the Appendix.
We assume periodicity of all variables in the $e_1=(1,0)$-direction and impose
 \begin{equation}
    \label{BC-T}
    \begin{aligned}
        T &= 1  & &\textnormal{ on }\gamma^- \\
        T &= 0  & &\textnormal{ on }\gamma^+
    \end{aligned}
\end{equation}
where $\gamma^-=\lbrace (y_1,y_2)\in \mathbb{R}^2\ \vert \ 0<y_1<\Gamma, y_2=h(y_1)\rbrace$ and $\gamma^+=\lbrace (y_1,y_2)\in \mathbb{R}^2\ \vert \ 0<y_1<\Gamma, y_2=1+h(y_1)\rbrace$ are the bottom and top boundary respectively.
Our analysis wants to capture the role of geometry and boundary conditions in the scaling laws for the Nusselt number $\Nu$, defined as
\begin{align}
    \label{introduction_def_nusselt}
    \Nu = \limsup_{t\to \infty} \frac{1}{t}\int_0^t \frac{1}{|\Omega|}\int_{\gamma^-}n\cdot \nabla T dS dt\,.
\end{align}
For a physical motivation of this definition of the Nusselt number between rough boundaries we refer to \cite{GD16} (Section 3).
In the last thirty years, the problem of deriving scaling laws for the Nusselt number in turbulent convection between two horizontal plates has been highly investigated both in experiments and numerical studies (cf.\cite{AGL09,chillaSchumacher2012} and references therein). 

\begin{figure}
\centering
   \includegraphics[scale=0.9]{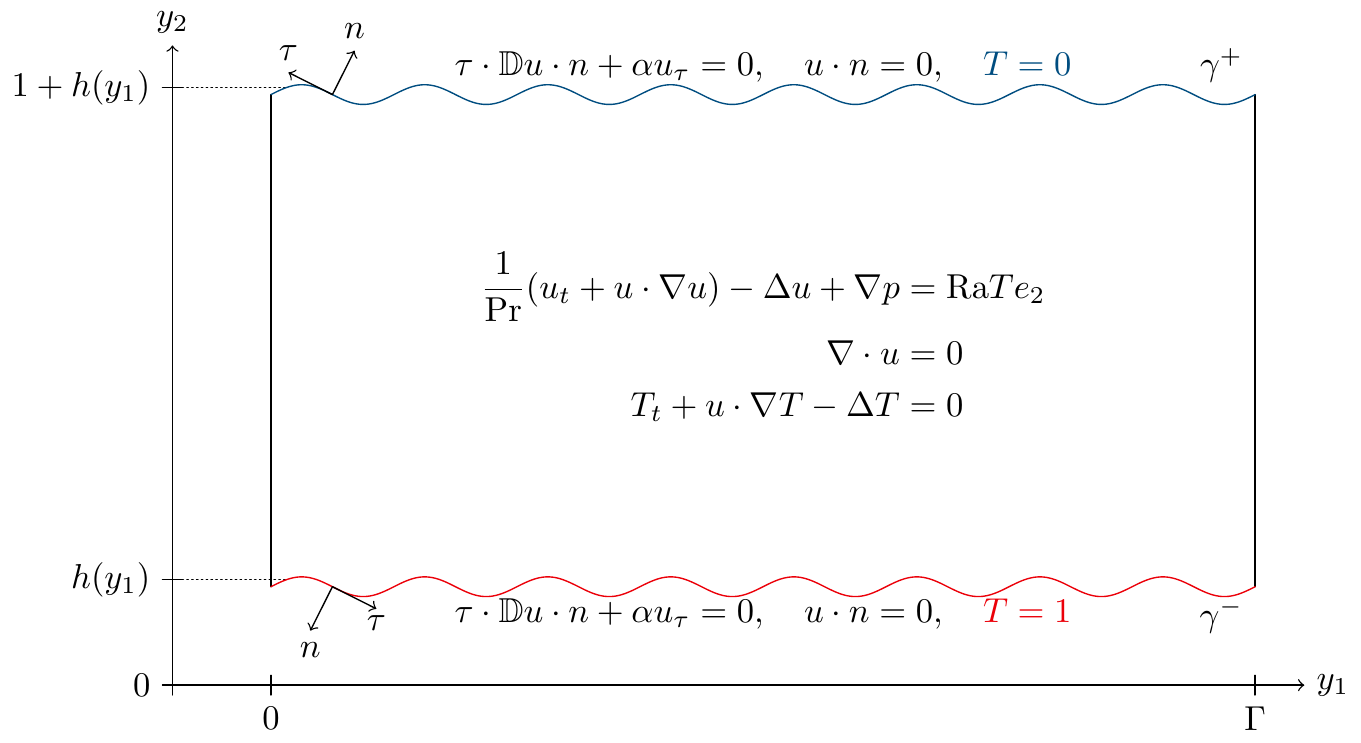}
    \caption{Illustration of the Rayleigh-B\'enard convection problem with Navier-slip boundary conditions and rough boundaries considered in this paper: \eqref{navierStokes}, \eqref{divergenceFreeIncompressibilityCondition}, \eqref{heatEquation}, \eqref{BC-T},\eqref{BC-u}.}
    \label{fig:overview}
\end{figure}
Intentionally, we did not yet specify the boundary conditions for $u$ at $\gamma^+$ and $\gamma^-$. Before doing that, we recall some important results concerning the scaling of the Nusselt number under \textit{no-slip} boundary conditions, which are the most studied for this problem. The no-slip boundary conditions in the case of flat boundaries (i.e. $h=0$) read
\begin{align*}
    u_1=0,\; u_2=0\quad \mbox{ at }\quad y_2=\{0,1\}\,.
\end{align*}
In this setting, Doering and Constantin in 1996 \cite{DC96} rigorously proved the upper bound $\Nu\lesssim \Ra^{\frac 12}$ in three-dimensions. From this seminal result, many works followed aiming at optimizing the $\Ra^{\frac 12}$-upper bound (see \cite{N21}, \cite{FAW21} and references therein).

In the infinite Prandtl number setting, a series of works \cite{CD99}, \cite{DOR06}, \cite{OS11} established the bound $\Nu\lesssim \Ra^{\frac 13}$ up to a logarithmic correction using the \textit{background field method}. This method, based on the decomposition of the temperature $T$ into a steady background profile $\tau$ and fluctuations around it, converts the problem of finding upper bounds for the Nusselt number into a variational problem: finding the background profile $\tau$ with minimal Dirichlet energy, satisfying a certain spectral condition. The limitation of this method in producing the optimal bound for the Nusselt number was shown in \cite{NO17}. The best bound for the Nusselt number for the Rayleigh-B\'enard convection under no-slip conditions and $\Pr=\infty$ is $\Nu\lesssim (\Ra\ln\ln\Ra)^{\frac 13}$ \cite{OS11} and is derived by combining the background field method with a delicate PDE analysis.

In the finite Prandtl number case, Choffrut, Otto and the second author of this paper in \cite{choffrutNobiliOttoUpperBoundsOnNusseltNumberAtFinitePrandtlNumber} improved the perturbative result of Wang \cite{WA08_boundOn} showing $\Nu\lesssim (\Ra\ln\Ra)^{\frac 13}$ for $\Pra\gtrsim (\Ra\ln\Ra)^{\frac 13}$ and the crossover to the bound $\Nu\lesssim (\Pr^{-1}\Ra\ln\Ra)^{\frac 12}$ for $\Pra\lesssim(\Ra\ln\Ra)^{\frac 13}$. 

The question around the "ultimate regime" in the Rayleigh-B\'enard convection problem was recently object of animated discussions \cite{D20} and it remains unclear whether at large Rayleigh number the scaling $\Nu\sim \Ra^{\frac 12}$ will prevail over the scaling $\Nu\sim \Ra^{\frac 13}$ or whether another scaling arises. For physical heuristics of the Kraichnan-Spiegel scaling $\Nu\sim \Ra^{\frac 12}$ and the Malkus scaling $\Nu\sim \Ra^{\frac 13}$ we refer the reader to Section 2.1 in \cite{N21}.

The methods proposed by Doering and Constantin in the nineties to produce bounds for the Nusselt number in boundary-driven convection continues to be fruitfully employed in other problems like "wall-to-wall optimal transport" \cite{tobasco2022}, Rayleigh-B\'enard convection with rotation \cite{DC01}, B\'enard-Marangoni convection \cite{FNW20} and internal heating \cite{AFW21} (see also \cite{Goluskin2016} and references therein). On a related note, let us mention that recently remarkable results have been achieved in the optimal designs for enstrophy-constrained transport \cite{DT19},\cite{kumar2022}.

The no-slip boundary conditions are the most used in theoretical and numerical studies, but whether or not they represent the most "realistic" conditions, is subject to debate (see \cite{N21} and references therein). In 2011 Doering and Whitehead \cite{whiteheadDoeringUltimateState} remarkably proved $\Nu\lesssim \Ra^{\frac{5}{12}}$ in the two-dimensional Rayleigh-B\'enard convection problem between flat horizontal boundaries, with free-slip boundary conditions, i.e.
$$u_2=0 \quad \mbox{ and } \quad\partial_{2}u_1=0 \quad \mbox{ at }\quad y_2=\{0,1\}\,. $$ 
This result rules out the $\Nu\sim \Ra^{\frac 12}$ scaling in this setting. But the question about the optimality of this upper bound remains along with the question whether the scaling $\Nu\sim \Ra^{\frac{5}{12}}$ carries a physical meaning. Motivated by this result, in \cite{drivasNguyenNobiliBoundsOnHeatFluxForRayleighBenardConvectionBetweenNavierSlipFixedTemperatureBoundaries} Drivas, Nguyen and the second author of this paper considered the two-dimensional Rayleigh-B\'enard convection problem with the following conditions on the horizontal plates:
\begin{equation}
    \label{Navier-slip-straight}
    u_2=0 \quad \mbox{ and } \quad \partial_{2}u_1=\frac{1}{L_s}u_1 \quad\mbox{ at }\quad y_2=\{0,1\}\,,
\end{equation}
where $L_s$ is the (constant) "slip-length". Under these assumptions the authors proved the \textit{interpolation bound}
\begin{align*}
    \Nu\lesssim \Ra^{\frac{5}{12}}+\frac{1}{L_s^2}\Ra^{\frac 12}\qquad \mbox{ in the regime } \qquad\Pr\gtrsim \frac{1}{L_s^2} \Ra^{\frac{3}{4}}\,.  
\end{align*}

Relatively few theoretical works have addressed the problem of the Nusselt number scaling in the case of \textit{rough boundaries} in Rayleigh-B\'enard convection: In \cite{SW11} Shishkina and Wagner developed an analytical model to estimate the Nusselt number deviations caused by the wall roughness and in \cite{WS15} the same authors performed direct numerical simulations.
In \cite{GD16} Goluskin and Doering considered the three-dimensional Rayleigh-B\'enard problem between rough boundaries. In particular, under the assumption that the profile $h$ (which may differ at the top and bottom boundaries) is a continuous and piecewise differentiable function of the horizontal coordinates and has squared-integrable gradients, the authors showed $\Nu\lesssim C(\|\nabla h\|_2^2)\Ra^{\frac 12}$.

The boundary conditions in \eqref{Navier-slip-straight} are the simplified version of the original \textit{Navier-slip} boundary conditions \cite{Navier1823}
\begin{equation}
    \label{BC-u}
    \begin{aligned}
        \tau \cdot \mathbb{D}u\cdot n +\alpha u_\tau &= 0 & &\textnormal{ on }\gamma^+\cup \gamma^-,\\
        u\cdot n &= 0 & &\textnormal{ on }\gamma^+\cup \gamma^-.
    \end{aligned}
\end{equation}
Here we denoted with $n$ and $\tau$ the outward unit normal and tangential vector to the boundary respectively, with $u_\tau = u\cdot \tau$ the tangential velocity and with $\mathbb{D}u_{ij} = \frac{1}{2}(\partial_i u_j +\partial_j u_i)$ the strain tensor. The space-dependent function $\alpha\geq 0$ is the friction coefficient and measures the tendency of the fluid to slip on the boundary. It may vary along $\gamma^-$ and $\gamma^+$. The problem is illustrated in Figure \ref{fig:overview}. We notice that the Navier-slip boundary conditions have been studied in a variety of problems in fluid mechanics. They have been considered in problems related to mixing \cite{huWu} as well as rotating systems \cite{dalibardVaret}.
In \cite{amroucheEscobedoGhosh} and \cite{acevdeoAmroucheConcaGhosh} the authors studied the Stokes operator with Navier-slip boundary conditions and the convergence to no-slip boundary conditions as $\alpha\to\infty$.
 
In this paper we want to generalize the result in \cite{drivasNguyenNobiliBoundsOnHeatFluxForRayleighBenardConvectionBetweenNavierSlipFixedTemperatureBoundaries}, considering the original \textit{Navier-slip} boundary conditions \eqref{BC-u}, hence allowing a non-constant (space dependent) friction coefficient. Furthermore, we allow a certain degree of \textit{roughness} at the upper and lower plates, characterized by the curvature $\kappa$, defined as $\frac{d}{d\lambda}\tau = \kappa n$, where $\lambda$ is the parameterization of the boundary by arc length and related to the height function via
\begin{align}
    \label{kappa_h_relation}
    \kappa(x_1) &= \pm \frac{h''(x_1)}{(1+(h'(x_1))^2)^\frac{3}{2}} \textnormal{ on } \gamma^\pm.
\end{align}

Our first, more general result is the following.
\begin{theorem}
\label{Lemma-Ra-One-Half-Bound}
Let $u$ and $T$ solve system \eqref{navierStokes}-\eqref{divergenceFreeIncompressibilityCondition}-\eqref{heatEquation}, with boundary conditions \eqref{BC-T} and \eqref{BC-u}. Assume $h \in W^{2,\infty}[0,\Gamma]$, $u_0\in L^2$, $\alpha>0$ and $\kappa$ satisfies
\begin{align}
    \label{ec}
    |\kappa| \leq 2\alpha + \frac{1}{4\sqrt{1+(h')^2}} \min\left\lbrace 1, \sqrt{\alpha} \right\rbrace
\end{align}
pointwise almost everywhere on $\gamma^-\cup\gamma^+$. 
Then there exists a constant $C>0$ depending only on $\|h'\|_{\infty}$ and $|\Omega|$ such that
\begin{align}\label{ub-R12}
    \Nu \leq C \left(\Ra^\frac{1}{2}+\|\kappa\|_\infty\right)
\end{align}
for all $\Ra\geq 1$.
\end{theorem}
This upper bound catches the classical Kraichnan-Spiegel scaling $\Ra^{\frac 12}$ and is uniform in the Prandtl number. The novelty of this (expected) result is to show explicitly the role played by the functions $\alpha$ and $\kappa$ in the bound. We notice that the bound holds under assumption \eqref{ec} which relates the magnitude of $|\kappa|$ and $\alpha$. Since $\kappa$ changes sign, this assumption is crucial to obtain the energy decay estimate (see Lemma \ref{lemmaEnergyDecay}). As we can see from \eqref{ub-R12}, the bound deteriorates as roughness (quantified by $\|\kappa\|_\infty$) increases.
As already remarked, the upper bound in Theorem \ref{Lemma-Ra-One-Half-Bound} seems to indicate the Kraichnan-Spiegel scaling law. However, we are neither able to confirm nor rule out this possibility as we can only produce upper bounds with our methods. In order to confirm the power $\frac 12$ of $\Ra$ in the scaling law, we would need to either derive lower bounds or provide an example of a flow that saturates the bound. Interestingly, in \cite{ZSVL17}, by means of direct numerical simulations, the authors are able to detect two regimes for two-dimensional Rayleigh-B\'enard convection over sinusoidal rough plates (varying height and wavelength of the rough elements) and no-slip conditions: in the first regime (up to $\Ra\sim 10^{9}$) data confirm the exponent $\frac 12$. Letting $\Ra$ grow further (between $\Ra\sim 10^{10}$ and $\Ra\sim 10^{12}$) the data seem to be consistent with the exponent $\frac 13$ (see Figure 2 in \cite{ZSVL17}). As the authors in \cite{ZSVL17} discuss, this picture is counterintuitive, as the system is supposed to become more bulk-dominated increasing $\Ra$. Nevertheless, our bounds do not rule out this possible intriguing picture, that the authors in \cite{ZSVL17} interpret as given by a "reverse role of boundary layer and bulk in the presence of roughness". Let us remark that in this paper we focus on Navier-slip boundary conditions, but the bound in Theorem \ref{Lemma-Ra-One-Half-Bound} would also hold under no-slip boundary conditions (see Remark \ref{noslip-remark}).  This result is not included in our paper since an upper bound catching the $Ra^{\frac 12}$-scaling was already rigorously proven by Goluskin and Doering \cite{GD16} for $H^1$-rough surfaces and no-slip boundary conditions.

The physical prediction based on experiments \cite{Roche20} is that roughness enhances heat transport and, in turn, increases the value of the Nusselt number. Interestingly the following theorem shows that, in the regime of large $\Pra$, if we assume higher regularity of the height function (loosely speaking this would mean that the surface is "less rough"), we obtain a (strictly smaller) bound $\Nu\lesssim \Ra^{\frac 37}$. This would rule out the Kraichnan-Spiegel ultimate regime.

\begin{theorem}
\label{main-theorem}
Let $u$ and $T$ solve the system \eqref{navierStokes}-\eqref{divergenceFreeIncompressibilityCondition}-\eqref{heatEquation}, with boundary conditions \eqref{BC-T} and \eqref{BC-u}.
Let $h\in W^{3,\infty}[0,\Gamma]$, $\alpha\in W^{1,\infty}(\gamma^-\cup\gamma^+)$ and $u_0\in W^{1,r}(\Omega)$ for some $r>2$. Set $\underline{\alpha}:=\min_{\gamma^-\cup\gamma^+}\alpha>0$. Then there exists a constant $0<\bar C<1$ such that for all $\alpha$ and $\kappa$ with
\begin{align}
    \label{theorem-condition-alpha+kappa-small}
    \|\alpha+\kappa\|_\infty \leq \bar C
\end{align}
the following bounds on the Nusselt number hold:
\begin{enumerate}
    \item\label{main-theorem-case-interpolation-kappa-leq-alpha}
        If $|\kappa|\leq \alpha$ on $\gamma^-\cup\gamma^+$, $\Pra\geq \underline{\alpha}^{-\frac{3}{2}}\Ra^\frac{3}{4}$ and $\Ra^{-\frac{1}{2}}\leq \underline{\alpha}$ then
        \begin{align}
            \label{main-theorem-bound-interpolation-kappa-leq-alpha}
            \Nu \leq C_\frac{1}{2} \|\alpha+\kappa\|_{W^{1,\infty}}^2 \Ra^\frac{1}{2} + C_\frac{5}{12}\Ra^\frac{5}{12}.
        \end{align}
    \item\label{main-theorem-case-interpolation-more-general-kappa}
        If $|\kappa|\leq 2\alpha + \frac{1}{4\sqrt{1+(h')^2}}\sqrt{\alpha}$ on $\gamma^-\cup\gamma^+$, $\Pra\geq \underline{\alpha}^{-\frac{3}{2}}\Ra^\frac{3}{4}$ and $\Ra^{-1}\leq \underline{\alpha}$ then
        \begin{align}
            \label{main-theorem-bound-interpolation-more-general-kappa}
            \Nu \leq C_\frac{1}{2} \sqrt{\underline{\alpha}}\|\alpha+\kappa\|_{W^{1,\infty}}^2 \Ra^\frac{1}{2} + C_\frac{5}{12} \underline{\alpha}^{-\frac{1}{12}} \Ra^\frac{5}{12}.
        \end{align}
    \item\label{main-theorem-case-3over7bound}
        If $|\kappa|\leq 2\alpha + \frac{1}{4\sqrt{1+(h')^2}}\sqrt{\alpha}$ on $\gamma^-\cup\gamma^+$ and $\Pra\geq \Ra^\frac{5}{7}$ then
        \begin{align}
            \label{main-theorem-bound-3over7bound}
            \Nu \leq C_\frac{3}{7} \Ra^\frac{3}{7}.
        \end{align}
\end{enumerate}
The constants are given by 
\begin{align*}
    C_\frac{1}{2} &= C(1+\|u_0\|_{W^{1,r}}^2)^{-1}
    \\
    C_\frac{5}{12} &= C \left(\|u_0\|_{W^{1,r}}+\|\dot\alpha\|_\infty+\|\dot\kappa\|_\infty+1\right)^\frac{1}{3}
    \\
    C_\frac{3}{7} &= C \left(\|\alpha+\kappa\|_{W^{1,\infty}}^2 + \underline{\alpha}^{-\frac{1}{2}} + \underline{\alpha}^{-\frac{1}{6}}(\|u_0\|_{W^{1,r}}+\|\dot\alpha\|_\infty+\|\dot\kappa\|_\infty+1)^\frac{1}{3}\right),
\end{align*}
where $\dot \alpha$ and $\dot \kappa$ are the derivatives of $\alpha$ and $\kappa$ along the boundary and $C>0$ denotes a constant depending only on the size of the domain $|\Omega|$, $\|h'\|_{\infty}$ and $r$.
\end{theorem}

First of all, we notice that our results in Theorem \ref{Lemma-Ra-One-Half-Bound} and Theorem \ref{main-theorem} also cover the case of flat-boundaries and generic friction coefficient $\alpha$. In fact, all results (i.e. \eqref{ub-R12}, \eqref{main-theorem-bound-interpolation-kappa-leq-alpha}, \eqref{main-theorem-bound-interpolation-more-general-kappa} and \eqref{main-theorem-bound-3over7bound}) hold setting $\kappa=0$.

We observe that Theorem \ref{main-theorem} only holds under the smallness assumption \eqref{theorem-condition-alpha+kappa-small}, while there is no such restriction in Theorem \ref{Lemma-Ra-One-Half-Bound}. If $\kappa$ and $\alpha$ are big, and related through \eqref{ec}, then the bound \eqref{ub-R12} holds.

Also note that Theorem \ref{main-theorem} requires the rather restrictive assumption $h\in W^{3,\infty}$ to ensure $\kappa\in W^{1,\infty}$, while Theorem \ref{Lemma-Ra-One-Half-Bound} assumes $h\in W^{2,\infty}$. The latter assumption is needed to control $\kappa \in L^{\infty}$ due to the boundary term arising in the energy balance \eqref{energy-balance}, while the former is necessary to get higher order estimates as can be seen in Lemma \ref{lemma_u_bounded_by_omega} or when estimating the terms arising from the vorticity balance in \eqref{Q-estimation-up}.

The interpolation bound \eqref{main-theorem-bound-interpolation-kappa-leq-alpha} coincides with the result in \cite{drivasNguyenNobiliBoundsOnHeatFluxForRayleighBenardConvectionBetweenNavierSlipFixedTemperatureBoundaries}, when $\kappa=0$ and $\alpha=\frac{1}{L_s}$. The main advantage of the interpolation results \eqref{main-theorem-bound-interpolation-kappa-leq-alpha} and \eqref{main-theorem-bound-interpolation-more-general-kappa} is in the regime of small $\|\alpha\|_{W^{1,\infty}}$ and $\|\kappa\|_{W^{1,\infty}}$. In fact, notice that if $\|\kappa\|_{W^{1,\infty}}\sim \Ra^{-\nu}$ and $\|\alpha\|_{W^{1,\infty}}\sim \Ra^{-\mu}$ (for a physical motivation of this scaling see the \hyperlink{kappaScalingTarget}{Appendix}), then \eqref{main-theorem-bound-interpolation-kappa-leq-alpha} translates into 
$$\Nu\lesssim \begin{cases} \Ra^{-2\mu+\frac 12}& \mbox{ if } \mu<\frac{1}{24}\\\Ra^{\frac{5}{12}}& \mbox{ if } \frac{1}{24}\leq\mu<\frac{1}{2}\end{cases}$$
and $\nu\geq \mu$. This means that the best bound is achieved in the case $\frac{1}{24}\leq\mu<\frac{1}{2}$ and $\Pr\gtrsim \Ra^{\frac 32\mu+\frac 34}$. Notice that, while \eqref{main-theorem-bound-interpolation-kappa-leq-alpha} holds only under the assumption $|\kappa|\leq \alpha$, the upper bound \eqref{main-theorem-bound-interpolation-more-general-kappa} holds under the weaker assumption $|\kappa|\leq 2\alpha + \frac{1}{4\sqrt{1+(h')^2}}\sqrt{\alpha}$ on $\gamma^-\cup\gamma^+$, allowing $|\kappa|>\alpha$ especially when $\alpha$ and $|\kappa|$ are very small.
We observe that in the regime of small $\alpha$ and $\kappa$, say $\|\alpha\|_{W^{1,\infty}},\|\kappa\|_{W^{1,\infty}}\leq 1$, the constant $C_\frac{5}{12}$ becomes independent of $\alpha$ and $\kappa$.
Finally we remark that if the condition $\Ra^{-1}\leq \underline{\alpha}$ for \eqref{main-theorem-bound-interpolation-more-general-kappa} is violated, then \eqref{ub-R12} yields a stricter bound.

Interestingly \eqref{main-theorem-bound-3over7bound} improves the upper bound \eqref{ub-R12} in the case of small $\|\alpha+\kappa\|_\infty$ and big $\Pra$, i.e. $\Pra\geq \Ra^\frac{5}{7}$. While the interpolation bound \eqref{main-theorem-bound-interpolation-kappa-leq-alpha} yields a better result if $\alpha\sim \Ra^{-\mu}$, \eqref{main-theorem-bound-3over7bound} provides the sharpest bound if $\alpha$ and $\kappa$ are independent of $\Ra$. Notice that, differently from the constant prefactor in \eqref{ub-R12}, the constant $C_{\frac 37}$ also depends on $\alpha$ and $\kappa$.

\section{Preliminaries}
We start with introducing some notation and facts we will use in the whole paper. 
We will often use that the tangential vector to the boundary $\tau$ can be written as $\tau = n^\perp=(-n_2,n_1)$, where $n$ is outward unit normal and $u_\tau = u\cdot \tau$. 
 The curvature of the boundaries $\kappa$ is given by $\frac{d}{d\lambda} \tau = \kappa n$, where $\lambda$ is the parameterization of the boundaries by arc length in $\tau$-direction. The problem is illustrated in Figure \ref{fig:overview}.

We will use the following notation to denote space-time averages
\begin{align*}
    \langle f \rangle &= \limsup_{t\to\infty}\frac{1}{t}\int_0^t \frac{1}{|\Omega|}\int_\Omega f \ dy\ dt,\\
    \langle f \rangle_{\gamma^-} &= \limsup_{t\to\infty}\frac{1}{t}\int_0^t \frac{1}{|\Omega|}\int_{\gamma^-} f \ dS\ dt,\\
    \langle f \rangle_{\gamma^-\cup\gamma^+} &= \limsup_{t\to\infty}\frac{1}{t}\int_0^t \frac{1}{|\Omega|}\int_{\gamma^-\cup\gamma^+} f \ dS\ dt.
\end{align*}
Moreover, on the set 
\begin{equation}\label{strip-av}
    \gamma(x_2)=\left\lbrace(y_1,y_2)\ \middle | \ 0<y_1<\Gamma, y_2=h(y_1)+x_2 \right\rbrace,
\end{equation}
where $0\leq x_2\leq 1$, illustrated in Figure \ref{fig:subdomains}, we define the average
\begin{align*}
    \langle f \rangle_{\gamma(x_2)} &=\limsup_{t\to\infty}\frac{1}{t}\int_0^t \frac{1}{|\Omega|}\int_{\gamma(x_2)} f\ dS\,dt.
\end{align*}

In what follows we will consider $0\leq T_0\leq 1$, implying
\begin{align}
    \label{maximum-principle}
    \|T\|_\infty \leq 1
\end{align}
for all times $t>0$, by the maximum principle for the temperature.

\begin{figure}
    \includegraphics[width=\textwidth]{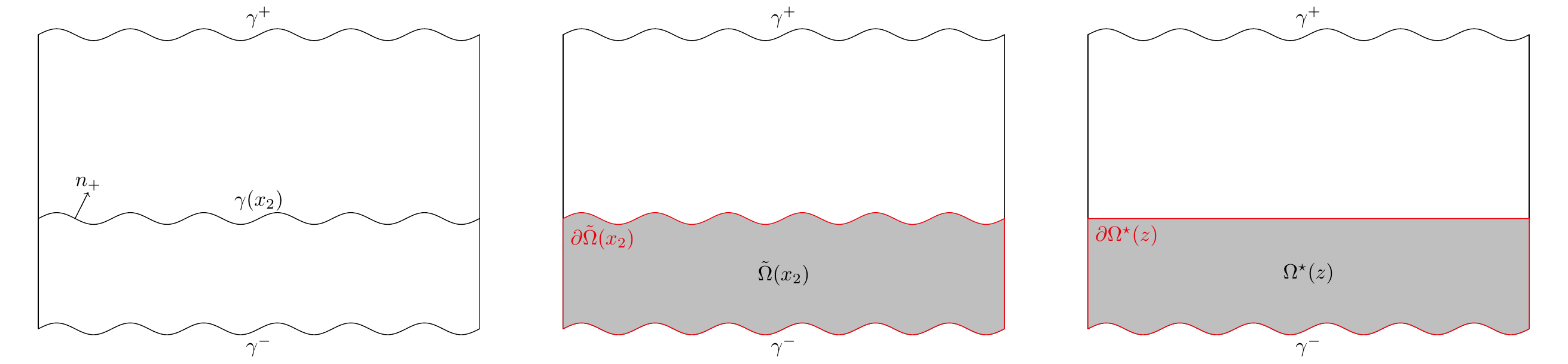}
    \caption{Illustrations of $\gamma(x_2)$, $\tilde\Omega(x_2)$ and $\Omega^\star(z)$}
    \label{fig:subdomains}
\end{figure}

We redefine the Nusselt number (see \eqref{introduction_def_nusselt}) adopting the more compact brackets notation:
\begin{definition}
    The Nusselt number is defined as
    \begin{equation}
        \Nu=\langle n \cdot \nabla T \rangle_{\gamma^-}.
    \end{equation}
\end{definition}
This number admits other equivalent representations that will be useful in our future arguments.
\begin{proposition}
\label{temp_review_prop_nusselt_identities}
The Nusselt number satisfies
\begin{align}
    \Nu={}&\la|\nabla T|^2\ra\label{nusselt-gradT}\\
    ={}& \langle  (uT - \nabla T)\cdot n_+ \rangle_{\gamma(x_2)} \label{nusselt-strip}\\
    \geq{}& \frac{1}{1+\max h -\min h}\langle (u_2-\partial_2)T \rangle\,.\label{nusselt-ineq}
\end{align}
where $n_+$ is the normal at the curve $\gamma(x_2)$ pointing in the same direction as $n$ on $\gamma^+$, illustrated in Figure \ref{fig:subdomains}.
\end{proposition}

\begin{proof}
Argument for \eqref{nusselt-gradT}: Testing the temperature equation with $T$, integrating by parts and using the the divergence-free condition and the boundary conditions for $T$, we obtain
\begin{equation*}
    \frac{1}{2}\frac{d}{dt}\|T\|_2^ 2 = -\|\nabla T\|_2^2 + \int_{\gamma^-} n \cdot \nabla T.
\end{equation*}
Taking the long-time averages and using \eqref{maximum-principle}, the maximum principle for the temperature, $T$ is universally bounded in time and we get
\begin{equation*}
    0=\lim_{t\to\infty} \frac{1}{2t} \frac{1}{|\Omega|} \left(\|T\|_2^2-\|T_0\|_2^2\right) 
    =-\langle |\nabla T|^2\rangle +\langle n\cdot \nabla T\rangle_{\gamma^-}.
\end{equation*}

Argument for \eqref{nusselt-strip}: For $0\leq x_2\leq 1$, we define the sets 
\begin{equation}
    \label{definition_OmegaTilde}
            \tilde\Omega(x_2)=\left\lbrace (y_1,y_2)\in \mathbb{R}^2\ \middle \vert\ 0<y_1<\Gamma,\ h(y_1)<y_2<h(y_1)+x_2 \right\rbrace,
\end{equation}
illustrated in Figure \ref{fig:subdomains}, and integrate the equation for $T$, obtaining
 \begin{align*}
            \int_{\tilde\Omega(x_2)} T_t &= \int_{\tilde\Omega(x_2)} \nabla \cdot(\nabla T -uT) 
            =\int_{\partial\tilde\Omega(x_2)} n\cdot (\nabla T - u T)\\
            &=\int_{\gamma^-} n_-\cdot (\nabla T - u T)+\int_{\gamma(x_2)} n_+\cdot (\nabla T - u T)\\
            &=\int_{\gamma^-} n_-\cdot \nabla T +\int_{\gamma(x_2)} n_+\cdot (\nabla T - u T)
\end{align*}
where we used the incompressibility condition and that $u\cdot n=0$ at $\gminus$.
Taking the long-time average we get
\begin{equation*}
    \langle n\cdot \nabla T\rangle_{\gamma^-}=\la (uT-\nabla T)\cdot n_+\ra_{\gamma(x_2)}.    
\end{equation*}
In particular, from this identity we infer that the Nusselt number is independent of $x_2$.

Argument for \eqref{nusselt-ineq}: Let us define the set $\Omega^\star(z)=\{y_2\leq z\}\cap \Omega$, as shown in Figure \ref{fig:nusselt-inequality}, and write
\begin{align*}
    \int_{\partial\Omega^{\star}(z)}n\cdot (\nabla T-uT)&=\int_{\gminus\cap\{y_2\leq z\}}n_-\cdot (\nabla T-uT)+\int_{\gplus\cap\{y_2\leq z\}}n_+\cdot (\nabla T-uT)
    \\&\qquad+\int_{\Omega\cap\{y_2=z\}}n_+\cdot (\nabla T-uT).
\end{align*}

\begin{figure}
    \begin{center}
        \includegraphics[width=0.5\textwidth]{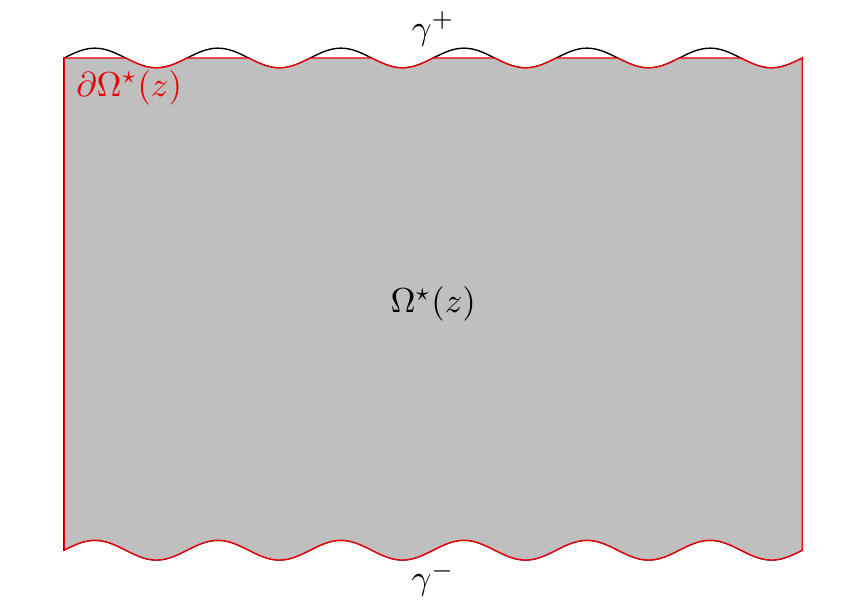}    
    \end{center}
    \caption{Illustration of $\Omega^\star(z)$}
    \label{fig:nusselt-inequality}
\end{figure}
Integrate the previous equation in $z$ between $\min h$ and $1+\max h$ and write
\begin{equation*}
    \int_{\min h}^{1+\max h}\int_{\partial\Omega^{\star}(z)}n\cdot (\nabla T-uT)\, dS\, dz=A+B+C,
\end{equation*}
where 
\begin{align*}
    A&=\int_{\min h}^{1+\max h}\int_{\gminus\cap\{y_2\leq z\}}n_-\cdot (\nabla T-uT)\, dS\, dz\\
    B&=\int_{\min h}^{1+\max h}\int_{\gplus\cap\{y_2\leq z\}}n_+\cdot (\nabla T-uT)\, dS\, dz\\
    C&=\int_{\min h}^{1+\max h}\int_{\Omega\cap\{y_2=z\}}n_+\cdot (\nabla T-uT)\, dS\, dz.
\end{align*}
We analyze the three terms separately: 

Term C: We notice that this term can be simply rewritten as
\begin{align*}
    C=\int_{\min h}^{1+\max h}\int_{\Omega\cap\{y_2=z\}} (\partial_2 T-u_2T)\, dS\, dz=\int_{\Omega} (\partial_2T-u_2T).
\end{align*}

Term B: This integral has a sign, in fact
\begin{align}
    \label{nusselt_change_different_profiles}
    B=\int_{\min h}^{1+\max h}\int_{\gplus\cap\{y_2\leq z\}}n_+\cdot \nabla T\, dS\, dz\leq 0,  
\end{align}
since, at $\gamma^+$ we have $n_+\cdot \nabla T\leq 0$.

Term A: We decompose the integral in A further:
\begin{align*}
    A=\left(\int_{\min h}^{\max h}+\int_{\max h}^{1+\max h}\right)\int_{\gminus\cap\{y_2\leq z\}}n_-\cdot (\nabla T-uT)\, dS\, dz=:A_1+A_2.
\end{align*}
Notice that
\begin{align*}
    A_2&=\int_{\max h}^{1+\max h}\int_{\gminus}n_-\cdot (\nabla T-uT)\, dS\, dz=\int_{0}^{1} dz\int_{\gminus}n_-\cdot (\nabla T-uT)\, dS 
    \\
    &=\int_{\gminus}n_-\cdot \nabla T.
\end{align*}
The term $A_1$ instead will be estimated as follows
\begin{align*}
    \int_{\min h}^{\max h}&\int_{\gminus\cap\{y_2\leq z\}}n_-\cdot (\nabla T-uT)\, dS\, dz
    \\
    &=\int_{\min h}^{\max h}\int_{\gminus\cap\{y_2\leq z\}}n_-\cdot \nabla T\, dS\, dz
    \leq\int_{\min h}^{\max h}\int_{\gminus}n_-\cdot \nabla T\, dS\, dz
    \\
    &= (\max h - \min h)\int_{\gminus}n_-\cdot \nabla T,
\end{align*}
where in the first inequality we used that $n_-\cdot \nabla T\geq 0$ at $\gminus$.

Putting all together, we obtain
\begin{align*}
    \int_{\min h}^{1+\max h}&\int_{\partial\Omega^{\star}}n\cdot (\nabla T-uT)\, dS\, dz
    \\
    &\leq (1+\max h - \min h)\int_{\gminus}n_-\cdot \nabla T+\int_{\Omega} (\partial_2T-u_2T).
\end{align*}
Taking the long-time average and observing that
\begin{align*}
    \limsup_{t\rightarrow \infty}\frac{1}{t}\int_0^t &\int_{\min h}^{1+\max h}\int_{\partial\Omega^{\star}(z)}n\cdot (\nabla T-uT)\ dS\ dz\ dt
    \\
    &=\limsup_{t\rightarrow \infty}\frac{1}{t}\int_0^t \int_{\min h}^{1+\max h}\int_{\Omega^{\star}(z)} (\Delta T-u\cdot \nabla T)\ dy\ dz\ dt
    \\
    &=\limsup_{t\rightarrow \infty}\frac{1}{t}\int_0^t\frac{d}{dt} \int_{\min h}^{1+\max h}\int_{\Omega^{\star}(z)} T \ dy\ dz\ dt 
    =0\,
\end{align*}
by the maximum principle for $T$, we have 
\begin{align*}
    0\leq (1+\max h-\min h)\la n_-\cdot \nabla T\, \ra_{\gminus}+\la\partial_2T-u_2T\ra
\end{align*}
implying
\begin{align*}
     \la u_2T-\partial_2T\ra\leq(1+\max h-\min h)\la n_-\cdot \nabla T\ra_{\gminus}.
\end{align*}
\end{proof}

\section{A-priori bounds}\label{Section-A-Priori-Bounds}
In this section we collect a-priori bounds on the energy and enstrophy of the solution $u$ and derive pressure estimates that will be used to prove the main result in the Section \ref{section-four}. 

\subsection{A-priori estimate for the velocity}
\begin{proposition}[Energy Balance]
Strong solutions of \eqref{navierStokes}, \eqref{divergenceFreeIncompressibilityCondition}, \eqref{BC-u} satisfy
\begin{align}
    \label{energy-balance}
    \frac{1}{2 \Pra} \frac{d}{dt}\|u\|_2^2 + \|\nabla u\|_2^2 + \int_{\gamma^+\cup \gamma^-} (2\alpha+\kappa) u_\tau^2= \Ra\int_{\Omega} T u_2.
\end{align}
\end{proposition}

\begin{proof}
The balance follows by testing the Navier Stokes equations with $u$, integrating by parts and observing that
\begin{align*}
    \int_\Omega u \cdot \nabla p
    =0
\end{align*}
and
\begin{align*}
    \int_\Omega u \cdot (u\cdot \nabla ) u 
    =-\int_\Omega u\cdot (u\cdot\nabla)u
\end{align*}
by the incompressibility and boundary conditions, and
\begin{align*}
    \int_\Omega u\cdot \Delta u
    &=-\|\nabla u\|_2^2 +\int_{\partial \Omega} n_i u_j\partial_i u_j
    \\
    &= -\|\nabla u\|_2^2 +2 \int_{\gamma^+\cup \gamma^-} u\cdot(\mathbb{D}u\ n) -\int_{\gamma^+\cup \gamma^-} n_iu_j \partial_j u_i\\
    &= -\|\nabla u\|_2^2 - 2\int_{\gamma^+\cup \gamma^-} \alpha u_{\tau}^2 -\int_{\gamma^+\cup \gamma^-} n\cdot (u\cdot \nabla) u
    \\
    &= -\|\nabla u\|_2^2 - \int_{\gamma^+\cup \gamma^-} (2\alpha+\kappa) u_\tau^2\,.
\end{align*}
In this identity we used the algebraic identity

\begin{equation*}
    n_iu_j\partial_iu_j=2u\cdot(\mathbb{D}u \,n)-n_iu_j\partial_ju_i,    
\end{equation*}
the Navier slip boundary conditions to deduce
\begin{equation*}
    u\cdot(\mathbb{D}u \,n)+\alpha u_{\tau}^2=0  
\end{equation*}
and the equality
\begin{equation}
    \label{id-kappaUtau2-1}
    n\cdot(u\cdot \nabla)u=\kappa u_{\tau}^2\,,
\end{equation}
proved in \eqref{appendix-proof-id-kappaUtau2} in the Appendix.

\end{proof}

Under a smallness assumption on $\kappa$, the second and third term on the left-hand side of \eqref{energy-balance} are positive and bounded from below by the $H^1$-norm of $u$ even though $(2\alpha+\kappa)$ might be negative on some parts of the boundary. This will be essential in what follows, especially in order to prove the energy decay and the main theorem.

\begin{lemma}
\label{lemma-H1-bounded-by-grad-and-bdry-terms}
Assume $\alpha \geq \underline{\alpha}$ almost everywhere on $\gamma^-\cup\gamma^+$ for some constant $\underline{\alpha}>0$ and $\kappa$ satisfies
\begin{align}
    \label{energy_condition_explicite_kappa}
    |\kappa| \leq 2\alpha + \frac{1}{4\sqrt{1+(h')^2}} \min\left\lbrace 1, \sqrt{\alpha} \right\rbrace
\end{align}
almost everywhere on $\gamma^-\cup\gamma^+$. Then
\begin{align}
    \label{H1-bounded-by-grad-and-bdry-terms}
    \frac{3}{4}\|\nabla u\|_2^2 + \int_{\gamma^-\cup\gamma^+} (2\alpha+\kappa) u_\tau^2 \geq \frac{1}{4}\min\lbrace 1, \underline{\alpha}\rbrace \|u\|_{H^1}^2.
\end{align}
\end{lemma}

\begin{proof}
In this proof we use the notation 
\begin{equation}
    \label{defKappaAlphaUpm}
    \begin{aligned}
        \kappa_-&=\kappa(y_1,h(y_1)), & \alpha_-&=\alpha(y_1,h(y_1)), & u_-&=u_\tau(y_1,h(y_1)),
        \\
        \kappa_+&=\kappa(y_1,1+h(y_1)), & \alpha_+&=\alpha(y_1,1+h(y_1)), & u_+&=u_\tau(y_1,1+h(y_1)),
    \end{aligned}
\end{equation}
which is the evaluation of the functions on the bottom or top boundary. Notice that, because of the symmetry of the domain, $\kappa_-=-\kappa_+$.

The idea of the proof is that if $\kappa_-$ is negative for some $y_1$, then $\kappa_+$ is positive and we can compensate by the fundamental theorem of calculus.

By the fundamental theorem of calculus, Young's and Hölder's inequality
\begin{equation}
    \label{fundamental-theorem-calc-estimate-for-u_minus}
    \begin{aligned}
        |u(y_1,y_2)|^2&\leq \left(u_- +\int_{h(y_1)}^{y_2} \partial_2 u(y_1,z) \ dz\right)^2
        \\
        &\leq (1+\epsilon) u_-^2 + (1+\epsilon^{-1}) \left(\int_{h(y_1)}^{y_2} \partial_2 u(y_1,z) \ dz\right)^2
        \\
        &\leq (1+\epsilon) u_-^2 + (1+\epsilon^{-1}) (y_2-h(y_1)) \|\partial_2 u \|_{L^2(\gamma^-,\gamma^+)}^2
    \end{aligned}
\end{equation}
for all $\epsilon>0$, where $\|\partial_2 u\|_{L^2(\gamma^-,\gamma^+)}^2=\int_{h(y_1)}^{1+h(y_1)}|\partial_2u(y_1,y_2)|^2 dy_2$ and analogously
\begin{align}
    \label{fundamental-theorem-calc-estimate-for-u_plus}
    |u(y_1,y_2)|^2
    &\leq (1+\epsilon) u_+^2 + (1+\epsilon^{-1}) (1+h(y_1)-y_2) \|\partial_2 u \|_{L^2(\gamma^-,\gamma^+)}^2.
\end{align}
Integrating \eqref{fundamental-theorem-calc-estimate-for-u_minus} and \eqref{fundamental-theorem-calc-estimate-for-u_plus} in $y_2$ one gets
\begin{equation}
    \label{u-estimated-by-bdry-and-grad}
    \begin{aligned}
        \|u\|_{L^2(\gamma^-,\gamma^+)}^2 &\leq (1+\epsilon) \min \lbrace u_-^2,u_+^2 \rbrace +\frac{1+\epsilon^{-1}}{2} \|\partial_2 u\|_{L^2(\gamma^-,\gamma^+)}^2
        \\
        &\leq (1+\epsilon) \max\lbrace \alpha_-^{-1}, \alpha_+^{-1}, 1 \rbrace \Big( \min\lbrace \alpha_-, \alpha_+\rbrace \sqrt{1+(h')^2}\min\lbrace u_-^2,u_+^2 \rbrace 
        \\
        &\qquad\qquad\qquad  + (2\epsilon)^{-1} \|\partial_2 u\|_{L^2(\gamma^-,\gamma^+)}^2\Big),
    \end{aligned}
\end{equation}
where in the last inequality we have smuggled in the factor $\sqrt{1+(h')^2}>1$. Next we claim that
\begin{equation}
    \label{energy-decay-kappa-removed-estimate}
    \begin{aligned}
        \min\lbrace\alpha_-,\alpha_+\rbrace \sqrt{1+(h')^2}u_{i}^2
        &
        \leq \frac{5}{16}\|\partial_2 u\|_{L^2(\gamma^-,\gamma^+)}^2 + (2\alpha_-+\kappa_-)\sqrt{1+(h')^2}u_-^2
        \\
        &\qquad+ (2\alpha_++\kappa_+)\sqrt{1+(h')^2}u_+^2
    \end{aligned}
\end{equation}
holds for either $i=+$ or $i=-$. Using \eqref{energy-decay-kappa-removed-estimate}, \eqref{u-estimated-by-bdry-and-grad} turns into
\begin{align*}
    \|u\|_{L^2(\gamma^-,\gamma^+)}^2
    &\leq (1+\epsilon) \max\lbrace \alpha_-^{-1}, \alpha_+^{-1}, 1 \rbrace \bigg[(2\alpha_-+\kappa_-)\sqrt{1+(h')^2}u_-^2 
    \\
    &\qquad +(2\alpha_++\kappa_+)\sqrt{1+(h')^2}u_+^2 +\left(\frac{5}{16}+(2\epsilon)^{-1}\right) \|\partial_2 u\|_{L^2(\gamma^-,\gamma^+)}^2\bigg].
\end{align*}
Integrating in $y_1$ and choosing $\epsilon=3$ yields
\begin{align*}
    \|u\|_2^2 &\leq 4 \max\lbrace \underline{\alpha}^{-1}, 1 \rbrace \left[ \int_{\gamma^-\cup\gamma^+} (2\alpha+\kappa) u_\tau^2 + \frac{1}{2}\|\partial_2 u\|_2^2\right]
\end{align*}
which implies the following bound for the $H^1$-norm
\begin{align*}
    \|u\|_{H^1}^2 &\leq 4 \max\lbrace \underline{\alpha}^{-1}, 1 \rbrace \left[ \int_{\gamma^-\cup\gamma^+} (2\alpha+\kappa) u_\tau^2 + \frac{3}{4} \|\nabla u\|_2^2\right].
\end{align*}

It is only left to show that \eqref{energy-decay-kappa-removed-estimate} holds. In order to prove the claim we distinguish between two cases.
\begin{itemize}
    \item
    If $|\kappa| \leq 2\alpha$, then both $2\alpha_\pm+\kappa_\pm>0$ and as $\kappa_-=-\kappa_+$ either $\kappa_-$ or $\kappa_+$ is non-negative. Assume first $\kappa_-\geq 0$. Then by these observations
    \begin{align*}
        \min\lbrace&\alpha_-,\alpha_+\rbrace\sqrt{1+(h')^2}u_{i}^2 
        \\&\leq 2\alpha_- \sqrt{1+(h')^2}u_{-}^2 
        \\
        &\leq (2\alpha_- + \kappa_-) \sqrt{1+(h')^2}u_{-}^2 + (2\alpha_+ + \kappa_+) \sqrt{1+(h')^2}u_{+}^2
    \end{align*}
    The case $\kappa_+>0$ follows similar with $u_+^2$ instead of $u_-^2$.
    \item
    Now assume that $\kappa_+<0$ and $|\kappa_+|>2\alpha_+$. The case $\kappa_-<0$ and $|\kappa_-|>2\alpha_-$ follows similarly by exchanging $+$ and $-$. Using \eqref{fundamental-theorem-calc-estimate-for-u_minus} with $y_2 = 1 + h(y_1)$, respectively \eqref{fundamental-theorem-calc-estimate-for-u_plus} with $y_2 = h(y_1)$, and observing that $\kappa_-=-\kappa_+>2\alpha_+>0$ it holds
    \begin{align*}
        - (2\alpha_-+\kappa_-)&\sqrt{1+(h')^2}u_-^2 -(2\alpha_++\kappa_+)\sqrt{1+(h')^2}u_+^2
        \\
        &= - (2\alpha_-+\kappa_-)\sqrt{1+(h')^2}u_-^2 + (\kappa_--2\alpha_+)\sqrt{1+(h')^2}u_+^2
        \\
        &\leq (\kappa_--2\alpha_+)(1+\epsilon^{-1})\sqrt{1+(h')^2}\|\partial_2 u\|_{L^2(\gamma^-,\gamma^+)}^2
        \\
        &\qquad - \left[2\alpha_-+2\alpha_+-\epsilon(\kappa_--2\alpha_+)\right]\sqrt{1+(h')^2}u_-^2.
    \end{align*}
    In order for the squared bracket to be positive we choose $\epsilon = \frac{\alpha_-+\alpha_+}{\kappa_--2\alpha_+}$ to get
    \begin{align*}
        - (2\alpha_-+\kappa_-)&\sqrt{1+(h')^2}u_-^2 -(2\alpha_++\kappa_+)\sqrt{1+(h')^2}u_+^2
        \\
        &\leq \left(\kappa_--2\alpha_+ + \frac{(\kappa_- - 2 \alpha_+)^2}{\alpha_- + \alpha_+}\right)\sqrt{1+(h')^2}\|\partial_2 u\|_{L^2(\gamma^-,\gamma^+)}^2
        \\
        &\qquad - \left[\alpha_-+\alpha_+\right]\sqrt{1+(h')^2}u_-^2.
    \end{align*}
    Then, as the smallness condition \eqref{energy_condition_explicite_kappa} implies
    \begin{align*}
        \kappa_- &= |\kappa_+| \leq 2\alpha_+ + \frac{1}{4\sqrt{1+(h')^2}} \min\left\lbrace 1, \sqrt{\alpha_+} \right\rbrace
        \\
        &\leq 2\alpha_+ + \frac{1}{4} \min\left\lbrace \frac{1}{\sqrt{1+(h')^2}}, \frac{\sqrt{\alpha_++\alpha_-}}{(1+(h')^2)^\frac{1}{4}} \right\rbrace,
    \end{align*}
    we get \vspace{-5pt}    \begin{align*}
        - (2\alpha_-+\kappa_-)&\sqrt{1+(h')^2}u_-^2 -(2\alpha_++\kappa_+)\sqrt{1+(h')^2}u_+^2
        \\
        &\leq  \frac{5}{16}\|\partial_2 u\|_{L^2(\gamma^-,\gamma^+)}^2 - \left[\alpha_-+\alpha_+\right]\sqrt{1+(h')^2}u_-^2,
    \end{align*}
    proving the claim.
\end{itemize}
\vspace{-5pt}
\end{proof}

\begin{remark}
Note that this Lemma can be improved: In fact for every $\kappa$ and $\alpha$ with $\kappa_\mp < \sqrt{\frac{2\alpha_\mp+2\alpha_\pm}{\sqrt{1+(h')^2}}+(2\alpha_\pm)^2}$,
where we use the notation \eqref{defKappaAlphaUpm}, it holds $\|\nabla u\|_2^2+\int (2\alpha+\kappa)u_\tau^2 > 0$. This implies energy decay in \eqref{energy-balance}.
\newline
Nevertheless, in order to simplify the estimates and improve readability we will work with assumption \eqref{energy_condition_explicite_kappa} instead. This choice will have no effects in terms of optimality of the bounds for the Nusselt number.
\end{remark}

We will use the energy balance and \eqref{H1-bounded-by-grad-and-bdry-terms} to prove the following decay estimate for the energy of $u$:
\begin{lemma}[Energy Decay]
\label{lemmaEnergyDecay}
Let the assumptions of Lemma \ref{lemma-H1-bounded-by-grad-and-bdry-terms} be satisfied. Then the energy of $u$ is bounded by
\begin{align}
    \label{Energy-decay}
    \|u\|_2^2 \leq e^{-\frac{1}{4}\min\lbrace 1, \underline{\alpha}\rbrace \Pra\phantom{.}  t}\|u_0\|_2^2 + 256 \max\left\lbrace 1, \underline{\alpha}^{-2}\right\rbrace|\Omega|\Ra^2.
\end{align}
\end{lemma}

\begin{proof}
By the energy balance \eqref{energy-balance}, Young's inequality and the maximum principle \eqref{maximum-principle}, we obtain
\begin{align}
    \label{energy-estimate-derivative}
    \frac{1}{2 \Pra} \frac{d}{dt}\|u\|_2^2 
    &\leq -\|\nabla u\|_2^2 - \int_{\gamma^+\cup \gamma^-} (2\alpha+\kappa) u_\tau^2 + \epsilon \|u_2\|_2^2 + \frac{4}{\epsilon}|\Omega|\Ra^2\,.
\end{align}
Plugging \eqref{H1-bounded-by-grad-and-bdry-terms} into \eqref{energy-estimate-derivative}, we find
\begin{align*}
    \frac{1}{2 \Pra} \frac{d}{dt}\|u\|_2^2
    &\leq -\left(\frac{1}{4}\min\lbrace 1, \underline{\alpha}\rbrace -\epsilon\right) \|u\|_2^2 + \frac{4}{\epsilon}|\Omega|\Ra^2\,,
\end{align*}
and choosing $\epsilon=\frac{1}{8}\min\lbrace 1, \underline{\alpha}\rbrace$ yields
\begin{align*}
    \frac{d}{dt}\|u\|_2^2 
    &\leq -\frac{1}{4}\min\lbrace 1, \underline{\alpha}\rbrace \Pra  \|u\|_2^2 + 64 \Pra \max \lbrace 1, \underline{\alpha}^{-1}\rbrace |\Omega|\Ra^2\,.
\end{align*}
Applying Gr\"onwall's inequality we obtain \eqref{Energy-decay}.
\end{proof}
Taking the long-time average of the energy balance \eqref{energy-balance}, using the fact that
    $$\limsup_{t\to\infty}\frac{1}{t}\int_0^t \frac{d}{dt}\|u\|_2^2 = \limsup_{t\to\infty}\frac{1}{t}\left(\|u\|_2^2-\|u_0\|_2^2\right)=0$$
thanks to the uniform bound \eqref{Energy-decay} one gets
\begin{align*}
    \langle |\nabla u|^2\rangle + \langle (2\alpha+\kappa) u_\tau^2\rangle_{\gamma^-\cup\gamma^+}= \Ra \langle T u_2 \rangle
\end{align*}
and observing that, by \eqref{nusselt-ineq}, 
\begin{equation*}
     \Ra\langle T u_2\rangle
    =  \Ra(\langle u_2T-\partial_2 T\rangle -1)
    \leq \Ra\left((1+\max h - \min h) \Nu-1\right),
\end{equation*}
we deduce the following
\begin{corollary}
\begin{align}
    \label{average-energy-balance}
    \langle|\nabla u|^2\rangle + \langle (2\alpha+\kappa) u_\tau^2\rangle_{\gamma^+\cup \gamma^-}\leq \Ra\left((1+\max h - \min h) \Nu-1\right).
\end{align}
\end{corollary}

\subsection{A-priori estimate for the vorticity}
We now introduce the two-dimensional vorticity $\omega=\nabla^{\perp}\cdot u$, where $\nabla^\perp = (-\partial_2,\partial_1)$. It is easy to see that $\omega$ satisfies the equation
\begin{equation}
    \label{vorticity-equation}
    \begin{aligned}
        \frac{1}{\Pra}\left(\omega_t +(u\cdot\nabla)\omega\right)-\Delta \omega &=\Ra \partial_1 T & \textnormal{ in }&\Omega,\\
        \omega &= -2(\alpha+\kappa)u_\tau & \textnormal{ on }&\gamma^+\cup \gamma^-.
    \end{aligned}
\end{equation}
Notice that the boundary term is deduced from the following computation 
\begin{align*}
    \omega&=\omega (\tau\cdot\tau)= \omega (-\tau_1 n_2 +\tau_2 n_1) = \tau_i n_j (-\partial_i u_j+\partial_j u_i)\\
    &=2 \tau \cdot \mathbb{D}u \cdot n- 2n\cdot(\tau\cdot \nabla) u = -2(\alpha+\kappa)u_\tau ,
\end{align*}
where we used that $\tau=(-n_2,n_1)$, the boundary conditions \eqref{BC-u} and the identity 
\begin{equation}
    \label{id-kappa}
    \kappa u_{\tau}=n\cdot (\tau\cdot \nabla )u,
\end{equation}
proved in \eqref{appendix-proof-id-kappa} in the Appendix.

\begin{proposition}\label{vorticity-balance}
The following vorticity balance holds
\begin{align*}
    \frac{1}{2\Pra}\frac{d}{dt}&\|\omega\|^2 +\frac{1}{\Pra}\frac{d}{dt}\int_{\gamma^-,\gamma^+} (\alpha+\kappa)u_\tau^2 +\|\nabla \omega\|_2^2 - \Ra\int_{\Omega} \omega \partial_1 T 
    \\
    &=
    -2 \int_{\gamma^-\cup\gamma^+} (\alpha+\kappa) u \cdot \nabla p
    - \frac{2}{3\Pra}\int_{\gamma^-\cup\gamma^+} (\alpha+\kappa) u\cdot (u\cdot\nabla) u 
    \\
    &\qquad -2 \Ra \int_{\gamma^-} (\alpha+\kappa)u_\tau n_1.
\end{align*}
\end{proposition}

\begin{proof}
Testing the vorticity equation with $\omega$ yields
\begin{equation}
    \label{vort-id}
    \frac{1}{2\Pra}\frac{d}{dt}\|\omega\|_2^2 = -\frac{1}{\Pra}\int_{\Omega} \omega (u\cdot \nabla)\omega + \int_{\Omega} \omega\Delta\omega +\Ra\int_{\Omega} \omega \partial_1 T
\end{equation}
Using the incompressibility condition, it is easy to see that the first term on the right-hand side vanishes since $u\cdot n=0$ at $\partial\Omega$.
In order to analyze the second term, we first notice that 
\begin{equation}
    \label{vort-balance-bdry-id-a}
    n\cdot\nabla\omega=\tau \cdot \Delta u=\frac{1}{\Pra}\tau \cdot u_t + \frac{1}{\Pra}\tau \cdot (u\cdot \nabla)u +\tau\cdot \nabla p -\Ra T n_1\,,
\end{equation}
where we used incompressibility in the first identity. Then, using the boundary conditions for the vorticity and temperature and \eqref{vort-balance-bdry-id-a}, we have 
\begin{align*}
    \int_{\Omega}\omega\Delta\omega 
    &= -\|\nabla \omega\|_2^2 +\int_{\gamma^-\cup\gamma^+} \omega n\cdot\nabla \omega\\
    &= -\|\nabla \omega\|_2^2 -2 \int_{\gamma^-\cup\gamma^+} (\alpha+\kappa)u_\tau n\cdot\nabla \omega\\
    &=-\|\nabla \omega\|_2^2 -\frac{2}{\Pra}\int_{\gamma^-\cup\gamma^+} (\alpha+\kappa)u_\tau \tau \cdot u_t - \frac{2}{\Pra}\int_{\gamma^-\cup\gamma^+} (\alpha+\kappa)u_\tau \tau \cdot (u\cdot \nabla)u \\
    &\qquad -2 \int_{\gamma^-\cup\gamma^+} (\alpha+\kappa)u_\tau \tau \cdot \nabla p + 2 \Ra \int_{\gamma^-\cup\gamma^+} (\alpha+\kappa)u_\tau T n_1\\
    &= -\|\nabla \omega\|_2^2 -\frac{1}{\Pra}\frac{d}{dt}\int_{\gamma^-\cup\gamma^+} (\alpha+\kappa)u_\tau^2 - \frac{2}{\Pra}\int_{\gamma^-\cup\gamma^+} (\alpha+\kappa) u \cdot (u\cdot \nabla)u \\
    &\qquad -2 \int_{\gamma^-\cup\gamma^+} (\alpha+\kappa) u \cdot \nabla p + 2 \Ra \int_{\gamma^-} (\alpha+\kappa)u_\tau  n_1\,.
\end{align*}
Plugging it into \eqref{vort-id} yields
\begin{align*}
    \frac{1}{2\Pra}\frac{d}{dt}\|\omega\|_2^2
    &=\int_{\Omega} \omega\Delta\omega +\Ra\int_{\Omega} \omega \partial_1 T
    \\
    &=-\|\nabla \omega\|_2^2 -\frac{1}{\Pra}\frac{d}{dt}\int_{\gamma^-\cup\gamma^+} (\alpha+\kappa)u_\tau^2 - \frac{2}{\Pra}\int_{\gamma^-\cup\gamma^+} (\alpha+\kappa) u\cdot (u\cdot\nabla) u 
    \\
    &\qquad -2 \int_{\gamma^-\cup\gamma^+} (\alpha+\kappa) u \cdot \nabla p -2 \Ra\int_{\gamma^-} (\alpha+\kappa)u_\tau n_1 +\Ra\int_{\Omega} \omega \partial_1 T\,.
\end{align*}
Let us observe that the term $\int(\alpha+\kappa) u\cdot (u\cdot\nabla) u$ is, in general, non-zero as the parameters $\alpha$ and $\kappa$ depend on the space variables.
\end{proof}

We now want to relate the $L^2$-norm of the vorticity with the $L^2$-norm of the enstrophy.

\begin{lemma}
\label{lemma_u_bounded_by_omega}
Let $2\leq q \leq p$. If $h\in W^{3,\infty}$ and $\omega\in W^{1,p}$
\begin{equation*}
    \begin{aligned}
        \|\nabla u\|_2^2 &= \|\omega\|_2^2 + \int_{\gamma^-\cup\gamma^+} \kappa u_\tau^2
        \\
        \| u \|_{W^{1,p}} &\leq C \left(\|\omega\|_p + \left(1+\|\kappa\|_\infty^{1+\frac{2}{q}-\frac{2}{p}}\right)\|u\|_q\right)
        \\
        \| u \|_{W^{2,p}} &\leq C \left(\|\nabla \omega\|_p + (1+\|\kappa\|_\infty)\| \omega\|_p +\left(1+\|\kappa\|_\infty^2 + \|\dot\kappa\|_\infty\right)\|u\|_p\right)
    \end{aligned}
\end{equation*}
holds, where the constant $C$ only depends on $p$, $|\Omega|$ and $\|h'\|_{\infty}$.
\end{lemma}

\begin{remark}
Note that in the case of flat boundaries the estimates simplify to
\begin{align*}
    \|\nabla u\|_2^2 = \|\omega\|_2^2, \qquad \|\nabla u\|_{W^{m,p}} \leq C \|\omega\|_{W^{m,p}}
\end{align*}
as proven in Lemma 7 in \cite{drivasNguyenNobiliBoundsOnHeatFluxForRayleighBenardConvectionBetweenNavierSlipFixedTemperatureBoundaries}.
\end{remark}

\begin{proof}
\leavevmode
\begin{itemize}
    \item
    Integrating by parts twice, we find
    \begin{equation}
        \label{appendix-gradu-equals-omega+boundary-1}
        \begin{aligned}
            \|\nabla u\|_2^2 &= \int_{\gamma^-\cup\gamma^+} u \cdot (n \cdot \nabla) u - \int u \cdot \Delta u 
            \\
            &= \int_{\gamma^-\cup\gamma^+} u \cdot (n \cdot \nabla) u + \int u^\perp \cdot \nabla \omega 
            \\
            &= \int_{\gamma^-\cup\gamma^+} u \cdot (n \cdot \nabla) u - \int_{\gamma^-\cup\gamma^+} u_\tau \omega + \|\omega\|_2^2,
        \end{aligned}
    \end{equation}
    where we used the identity
     \begin{align*}
        \nabla^\perp \omega = \begin{pmatrix}\partial_2^2 u_1 -\partial_1\partial_2 u_2\\-\partial_1\partial_2 u_1 +\partial_1^2 u_2\end{pmatrix} = \Delta u,
    \end{align*}
    due to incompressibility.
    Next notice that $\tau_i\tau_j + n_in_j = \delta_{ij}$. Therefore the second boundary term of the right-hand side of \eqref{appendix-gradu-equals-omega+boundary-1} can be rewritten as
    \begin{equation}
        \label{appendix-gradu-equals-omega+boundary-2}
        \begin{aligned}
            -\int_{\gamma^-\cup\gamma^+} u_\tau \omega &= -\int_{\gamma^-\cup\gamma^+} u_\tau \tau \cdot (\tau\cdot \nabla^\perp) u - \int_{\gamma^-\cup\gamma^+} u_\tau n \cdot (n\cdot \nabla^\perp) u
            \\
            &= -\int_{\gamma^-\cup\gamma^+} u \cdot (n\cdot \nabla) u + \int_{\gamma^-\cup\gamma^+} n \cdot (u\cdot \nabla) u,
        \end{aligned}
    \end{equation}
    where in the last identity we used that $\tau\cdot \nabla^\perp = -\tau^\perp \cdot \nabla=n \cdot\nabla$ and $u_\tau n\cdot\nabla^\perp= -u_\tau\cdot n^\perp \cdot\nabla = -u_\tau \tau  \cdot\nabla=-u\cdot\nabla$.
    The first term on the right-hand side of \eqref{appendix-gradu-equals-omega+boundary-1} cancels with the first term on the right-hand side of \eqref{appendix-gradu-equals-omega+boundary-2}, implying
    \begin{align*}
        \|\nabla u\|_2^2 = \int_{\gamma^-\cup\gamma^+} n\cdot (u\cdot \nabla) u + \|\omega\|_2^2
    \end{align*}
    Finally using 
    \begin{align}
        \label{id-kappaUtau2-2}
        n\cdot (u\cdot \nabla) u = \kappa u_\tau^2
    \end{align}
    on $\gamma^-\cup\gamma^+$, which is proven in \eqref{appendix-proof-id-kappaUtau2} in the Appendix, yields the claim.
    \item 
    Let $\phi$ be the stream function of $u$, i.e. $\nabla^\perp\phi=u$, then
    \begin{align*}
        \Delta \phi &= \omega\\
        \phi\vert_{\gamma^\pm} &= \phi_\pm
    \end{align*}
    with constants $\phi_+$ and $\phi_-$ and without loss of generality set $\phi_-=0$. We can calculate $\phi_+$ by
    \begin{align*}
        \phi_+ &= \frac{1}{|\Omega|} \int_0^\Gamma \phi_+ dy_1 = \frac{1}{|\Omega|} \int_0^\Gamma \left[\phi_- + \int_{h(y_1)}^{1+h(y_1)} \partial_2 \phi\ dy_2\right] dy_1 = -\frac{1}{|\Omega|}\int_{\Omega}u_1 \ dy.
    \end{align*}
    Therefore $\bar \phi = \phi + (y_2-h(y_1))\frac{1}{|\Omega|}\int_{\Omega} u_1\ dy$ solves
    \begin{equation*}
        \begin{aligned}
            \Delta \bar\phi &= \omega - h''\frac{1}{|\Omega|}\int_{\Omega} u_1 \ dy
            \\
            \bar\phi\vert_{\gamma^\pm} &= 0.
        \end{aligned}
    \end{equation*}
    In order to flatten the boundary we introduce the change of variables
    \begin{align*}
        x=\Phi(y) = \begin{pmatrix}y_1\\y_2 - h(y_1)\end{pmatrix}, \qquad y=\Psi(x) = \begin{pmatrix}x_1\\x_2 + h(x_1)\end{pmatrix}.
    \end{align*}
    Here and in the rest of the paper $C>0$ denotes a constant that possibly depends on $\|h'\|_{\infty}$, the size of the domain $|\Omega|$ and the Sobolev exponent and may change from line to line. Note that
    \begin{align*}
        \|h'\|_\infty \leq C,\qquad \|h''\|_\infty \leq C\|\kappa\|_\infty,\qquad \|h'''\|_\infty\leq C (\|\dot\kappa\|_\infty+\|\kappa\|_\infty),
    \end{align*}
    and for $\xi(x)=\chi(\Psi(x))=\chi(y)$ one has
    \begin{equation}
        \label{changeOfVariablesDerivative}
        \begin{aligned}
            \|\nabla \xi\|_p&\leq C\|\nabla\chi\|_p
            \\
            \|\nabla^2\xi\|_p&\leq C(\|\nabla^2\chi\|_p+\|\kappa\|_\infty \|\nabla \chi\|_p)
            \\
            \|\nabla^3\xi\|_p &\leq C\left(\|\nabla^3\chi\|_p+\|\kappa\|_\infty\|\nabla^2\chi\|_p+(\|\dot\kappa\|_\infty+\|\kappa\|_\infty)\|\nabla \chi\|_p\right)
        \end{aligned}
    \end{equation}
    and analogous for the transformation in the other direction.
    Then
    \begin{equation}
        \label{elliptic-regularity-pde-tilde-phi}
        \begin{aligned}
            \tilde L\tilde \phi&=\tilde f & \textnormal{ in }&[0,\Gamma]\times [0,1]
            \\
            \tilde \phi &= 0  & \textnormal{ on }&[0,\Gamma]\times \lbrace x_2=0\rbrace \cup \lbrace x_2=1\rbrace,
        \end{aligned}
    \end{equation}
    where $\tilde L \tilde \phi=\sum_{i,j} \partial_{x_i}(\tilde a_{i,j}\partial_{x_j}\tilde \phi(x))$ with $\tilde a_{1,1}=1$, $\tilde a_{1,2}=\tilde a_{2,1}=-h'$ and $\tilde a_{2,2}=1+(h')^2$, $\tilde \phi(x)=\bar \phi(\Psi(x))$ and $\tilde f=\tilde\omega-h''\frac{1}{|\Omega|}\int_{\Omega} u_1 dy$ with $\tilde \omega (x) = \omega(\Psi(x))$. As this operator is elliptic we get
    \begin{align}
        \label{appendix-straightened-elliptic-reg-1}
        \|\tilde \phi\|_{W^{2,p}}\leq C \|\tilde f\|_p
    \end{align}
    for some constant $C>0$ depending only on $p$, $|\Omega|$ and $\|h'\|_{\infty}$. Using Hölder's inequality and the estimates for the change of variables \eqref{changeOfVariablesDerivative}, \eqref{appendix-straightened-elliptic-reg-1} becomes
    \begin{align}
        \label{elliptic-regularity-phi-W2p}
        \|\tilde\phi\|_{W^{2,p}} \leq C\|\tilde f\|_p \leq C \|\omega\|_p + C \|\kappa\|_\infty \|u\|_p.
    \end{align}
    Going back to the definition of $\bar \phi$, we find that \eqref{changeOfVariablesDerivative}, Hölder's inequality and \eqref{elliptic-regularity-phi-W2p} yield
    \begin{equation}
        \label{ellptic-regularity-gradu-by-omega}
        \begin{aligned}
            \|\nabla u\|_p &= \|\nabla^2\phi\|_p \leq \|\nabla^2\bar \phi\|_p + C\|h''\|_\infty \|u_1\|_1
            \\
            &
            \leq C \left(\|\nabla^2 \tilde \phi\|_p + \|\kappa\|_\infty \|\nabla \tilde\phi\|_p + \|h''\|_\infty \|u_1\|_1\right)
            \\
            &
            \leq C \left(\|\omega\|_p+ \|\kappa\|_\infty \|u\|_p\right),
        \end{aligned}
    \end{equation}
    implying
    \begin{align}
        \label{elliptic-regularity-u-in-Wonep-bound}
        \|u\|_{W^{1,p}} &\leq C \left(\|\omega\|_p+(1+ \|\kappa\|_\infty) \|u\|_p\right).
    \end{align}
    In order to estimate the $L^p$-norm of $u$ by the $L^q$-norm use interpolation and Young's inequality to get
    \begin{align*}
        \|u\|_p \leq C \|\nabla u\|_p^\theta \|u\|_q^{1-\theta} + C \|u\|_q \leq \epsilon \theta C \|\nabla u\|_p + C \left( (1-\theta)\epsilon^{-\frac{\theta}{1-\theta}}+1\right)\|u\|_q 
    \end{align*}
    for $\frac{1}{p}=\theta \left(\frac{1}{p}-\frac{1}{2}\right)+\frac{1-\theta}{q}$ and all $\epsilon>0$. Then choosing $\epsilon^{-1}=(1+\|\kappa\|_\infty)$ and plugging in $\theta = \frac{2(q-p)}{2(q-p)-pq}$
    \begin{align*}
        (1+\|\kappa\|_\infty) \|u\|_p &\leq C \|\nabla u\|_p + C\left(1+\|\kappa\|_\infty^{\frac{1}{1-\theta}}\right)\|u\|_q 
        \\
        &\leq C \|\nabla u\|_p + C\left(1+\|\kappa\|_\infty^{1+\frac{2}{q}-\frac{2}{p}}\right)\|u\|_q
    \end{align*}
    proving the claim.

    \item In order to prove the $W^{2,p}$ bound notice that by \eqref{elliptic-regularity-pde-tilde-phi} $\hat \phi = \partial_{x_1}\tilde \phi$ solves
    \begin{equation*}
        \begin{aligned}
            \tilde L\hat\phi&=\hat f & \textnormal{ in }&[0,\Gamma]\times [0,1]
            \\
            \hat\phi &= 0  & \textnormal{ on }& \lbrace x_2=0\rbrace \cup \lbrace x_2=1\rbrace,
        \end{aligned}
    \end{equation*}
    with $\hat f = \partial_{x_1}\tilde\omega + h'''\frac{1}{|\Omega|}\int_{\Omega} u_1 dy + \partial_{x_1}(h''\partial_{x_2}\tilde \phi)+\partial_{x_2}(h''\partial_{x_1}\tilde \phi) - 2h'h'' \partial_{x_2}^2 \tilde \phi \in L^p$. Again using elliptic regularity and Hölder's inequality we find
    \begin{equation}
        \label{elliptic-regularity-ddx1W2}
        \begin{aligned}
            \|\partial_{x_1}\tilde \phi\|_{W^{2,p}} &= \|\hat\phi\|_{W^{2,p}}\leq C \|\hat f\|_p
            \\
            &\leq C\left(\|\nabla \tilde\omega\|_p + \|\dot\kappa\|_\infty (\|u\|_p + \|\nabla \tilde\phi\|_p) + \|\kappa\|_\infty \|\nabla^2\tilde\phi\|_p\right).
        \end{aligned}
    \end{equation}
    In order to estimate the missing term $\partial_{x_2}^3\tilde \phi$ notice as $h'$ is independent of $x_2$
    \begin{equation}
        \label{elliptic-regularity-dx2hoch3}
        \begin{aligned}
            \partial_{x_2}^3\tilde \phi &= \frac{1}{1+(h')^2}\partial_{x_2} \left(\partial_{x_2}((1+(h')^2)\partial_{x_2}\tilde \phi\right) 
            \\
            &= \frac{1}{1+(h')^2}\partial_{x_2} \left(\tilde L\tilde\phi-\partial_{x_1}^2\tilde \phi +\partial_{x_1}(h'\partial_{x_2}\tilde \phi ) + \partial_{x_2} (h'\partial_{x_1}\tilde \phi)\right)
            \\
            &= \frac{1}{1+(h')^2}\partial_{x_2} \left(\tilde f - \partial_{x_1}^2 \tilde \phi + \partial_{x_1} (h'\partial_{x_2}\tilde \phi)+\partial_{x_2}(h'\partial_{x_1}\tilde \phi)\right).
        \end{aligned}
    \end{equation}
    Taking the norm in \eqref{elliptic-regularity-dx2hoch3} we find
    \begin{equation}
        \label{appendix-stream-function-ddx2hoch3}
        \begin{aligned}
            \|\partial_{x_2}^3\tilde \phi\|_{p}&\leq C \left( \|\nabla \tilde f\|_p + \|\kappa\|_\infty \|\partial_{x_2}^2\tilde\phi\|_p + \|\partial_{x_1}\tilde \phi\|_{W^{2,p}} \right).
        \end{aligned}
    \end{equation}
    Combining \eqref{elliptic-regularity-ddx1W2} and \eqref{appendix-stream-function-ddx2hoch3}
    \begin{align*}
        \|\tilde\phi\|_{W^{3,p}}\leq C \left(\|\nabla \tilde f\|_p + \|\kappa\|_\infty \|\nabla^2\tilde\phi\|_p +\|\nabla \tilde\omega\|_p + \|\dot\kappa\|_\infty (\|u\|_p + \|\nabla \tilde\phi\|_p)  \right).
    \end{align*}
    Using Hölder's inequality we find
    \begin{align*}
        \|\nabla \tilde f\|_p \leq C(\|\nabla \tilde \omega\|_p + \|\dot\kappa\|_\infty \|u\|_p),
    \end{align*}
    which together with \eqref{elliptic-regularity-phi-W2p} yields
    \begin{equation}
        \label{elliptic-regularity-phi-W3p}
        \begin{aligned}
            \|\tilde\phi\|_{W^{3,p}}&\leq C \left(\|\nabla \tilde\omega\|_p + \|\kappa\|_\infty \|\omega\|_p+ \|\kappa\|_\infty^2 \|u\|_p + \|\dot\kappa\|_\infty (\|u\|_p + \|\nabla \tilde\phi\|_p) \right)
            \\
            &\leq  C \left(\|\nabla \omega\|_p + \|\kappa\|_\infty \|\omega\|_p+ (\|\kappa\|_\infty^2 + \|\dot\kappa\|_\infty) \|u\|_p \right),
        \end{aligned}
    \end{equation}
    where in the last inequality we used
    \begin{align}
        \label{elliptic-regularity-gradphi}
        \|\nabla\tilde\omega\|_p\leq C \|\nabla\omega\|_p, \qquad \|\nabla\tilde\phi\|_p\leq C\|\nabla\bar\phi\|_p \leq C (\|\nabla\phi\|_p+\|u_1\|_1)\leq C\|u\|_p.
    \end{align}
    By the definitions and the change of variables estimate \eqref{changeOfVariablesDerivative} and Hölder's inequality one gets
    \begin{align*}
        \|\nabla^2 u\|_p &= \|\nabla^3 \phi\|_p\leq \|\nabla^3\bar \phi\|_p + \|h'''\|_\infty \frac{1}{|\Omega|}\int |u_1| \\
        &\leq C\left(\|\nabla^3 \tilde\phi\|_p + \|\kappa\|_\infty \|\nabla^2 \tilde\phi\|_p + (\|\dot\kappa\|_\infty+\|\kappa\|_\infty)\|\nabla\tilde\phi\|_p + \|\kappa\|_\infty\|u\|_p \right)
        \\
        &\leq C\left(\|\nabla \omega\|_p +\|\kappa\|_\infty\|\omega\|_{p} + (\|\kappa\|_\infty+\|\kappa\|_\infty^2+\|\dot\kappa\|_\infty) \|u\|_p\right),
    \end{align*}
    where in the last inequality we used \eqref{elliptic-regularity-phi-W3p}, \eqref{elliptic-regularity-phi-W2p} and \eqref{elliptic-regularity-gradphi}.
    Finally using the $W^{1,r}$-bound for $u$, \eqref{elliptic-regularity-u-in-Wonep-bound}, and Young's inequality yields the claim.
\end{itemize}
\end{proof}

The next result concerns a crucial $L_t^{\infty}L^p_x-$bound for the vorticity.
\begin{lemma}\label{LemmaVorticityBound}
Let $p\in (2,\infty)$ and assume that the conditions of Lemma \ref{lemma-H1-bounded-by-grad-and-bdry-terms} are satisfied. Then there exists a constant $C$ depending only on $p$, $|\Omega|$ and $\|h'\|_{\infty}$ such that
\begin{align*}
    \|\omega\|_{p} &\leq  C\left[ \|\omega_0\|_p  +\left(1+\|\alpha+\kappa\|_\infty^{\frac{2(p-1)}{p-2}}\right)\|u_0\|_2+ C_{\alpha,\kappa}\Ra\right],
\end{align*}
where $C_{\alpha,\kappa}= \left(1+\|\alpha+\kappa\|_\infty^{\frac{2(p-1)}{p-2}}\right)\max\left\lbrace 1, \underline{\alpha}^{-1}\right\rbrace$.
\end{lemma}

\begin{proof}
Fix an arbitrary time $\bar t>0$ and decompose the solution $\omega$ to \eqref{vorticity-equation} as
$$\omega=\bar\omega_\pm +\tilde \omega_\pm $$
where $\tilde \omega_\pm$ solves
\begin{equation*}
    \begin{aligned}
        \frac{1}{\Pra} (\partial_t \tilde\omega_\pm + u\cdot \nabla \tilde \omega_\pm)-\Delta \tilde \omega_\pm &= \Ra \partial_1 T  &\textnormal{ in } &\Omega\\
        \tilde \omega_\pm &= \pm \Lambda &\textnormal{ on } &\gamma^+\cup\gamma^-\\
        \tilde \omega_{\pm,0} &= \pm |\omega_0| &\textnormal{ in } &\Omega\,,
    \end{aligned}
\end{equation*}
with $\Lambda= 2 \|(\alpha+\kappa)u_\tau\|_{L^\infty([0,\bar t]\times\lbrace\gamma^+\cup\gamma^-\rbrace)}$ and the difference $\bar\omega_\pm = \omega -\tilde \omega_\pm$ solves
\begin{equation*}
    \begin{aligned}
        \frac{1}{\Pra} (\partial_t \bar\omega_\pm + u\cdot \nabla \bar \omega_\pm)-\Delta \bar \omega_\pm &= 0  &\textnormal{ in } &\Omega\\
        \bar \omega_\pm &= -2(\alpha+\kappa) u_\tau \mp \Lambda  &\textnormal{ on } &\gamma^+\cup\gamma^-\\
        \bar \omega_{\pm,0} &= \omega_0 \mp |\omega_0| &\textnormal{ in } &\Omega\,.
    \end{aligned}
\end{equation*}
Since the boundary and the initial values have a sign, i.e. $\bar{\omega}_+\leq 0$ on $\gplus\cup\gminus$, $\bar{\omega}_{+,0}\leq 0$ in $\Omega$ and $\bar{\omega}_-\geq 0$ on $\gplus\cup\gminus$, $\bar{\omega}_{-,0}\geq 0$ in $\Omega$, then, by the maximum principle, $\omega -\tilde \omega_+ =\bar \omega_+ \leq 0 $ and $0\leq \bar \omega_- = \omega - \tilde \omega_-$ yielding $\tilde \omega_- \leq \omega \leq \tilde \omega_+$.
In particular
\begin{align}
    |\omega| \leq \max \lbrace |\tilde\omega_-|,|\tilde \omega_+| \rbrace\,.
\end{align}
Hence, it remains to find upper bounds for $|\tilde\omega_-|$ and $|\tilde \omega_+|$. By symmetry, it suffices to show an upper bound for $\tilde \omega_+$. We divide the proof in three steps:

\medskip
\textbf{Step 1}: Omitting the indices, we define
\begin{align*}
    \hat \omega = \tilde \omega - \Lambda,
\end{align*}
then $\hat{\omega}$ satisfies
\begin{align*}
    \frac{1}{\Pra} (\partial_t \hat\omega + u\cdot \nabla \hat \omega)-\Delta \hat \omega &= \Ra \partial_1 T  &\textnormal{ in } &\Omega\\
    \hat \omega &= \Lambda-\Lambda = 0 &\textnormal{ on } &\gamma^+\cup\gamma^-\\
    \hat \omega_{0} &= |\omega_0|-\Lambda &\textnormal{ in } &\Omega\,.
\end{align*}
Testing the equation with $\hat{\omega}|\hat{\omega}|^{p-2}$ we obtain
\begin{align*}
    \frac{1}{p\Pra}\frac{d}{dt} \|\hat\omega\|_p^p 
    &= -(p-1)\int_\Omega |\nabla \hat\omega|^2 |\hat\omega|^{p-2} -\Ra \int_\Omega T  \partial_1(|\hat{\omega}|^{p-2}\hat\omega).
\end{align*}
Using $\|T\|_{\infty}=1$, Young's and Hölder's inequality, we estimate the second term of the right-hand side as
\begin{align*}
    \left|\Ra\int_\Omega T  \partial_1(|\hat{\omega}|^{p-2}\hat\omega)\right|
    \leq \frac{p-1}{2}\left(\Ra^2 |\Omega|^{\frac{2}{p}} \|\hat\omega\|_p^{p-2} + \int_\Omega |\nabla\hat\omega|^2 |\hat\omega|^{p-2}\right).  
\end{align*}
Then 
\begin{align*}
    \frac{1}{p\Pra}\frac{d}{dt} \|\hat\omega\|_p^p 
    &\leq \frac{p-1}{2}\left( \Ra^2 |\Omega|^\frac{2}{p} \|\hat\omega\|_p^{p-2} -\int_\Omega |\nabla \hat\omega|^2 |\hat\omega|^{p-2} \right)
    \\
    &=\frac{p-1}{2}\left( \Ra^2 |\Omega|^\frac{2}{p} \|\hat\omega\|_p^{p-2} -\frac{4}{p^2}\big\|\nabla |\hat\omega|^\frac{p}{2}\big\|_2^2 \right).
\end{align*}
By the Poincar\'e estimate applied to the second term of the right-hand side (remember that $\hat{\omega}$ vanishes at the boundary by definition), we obtain
  \begin{align*}
    \frac{1}{p\Pra}\frac{d}{dt} \|\hat\omega\|_p^p \leq \frac{p-1}{2} \Ra^2 |\Omega|^\frac{2}{p} \|\hat\omega\|_p^{p-2} - 2\frac{p-1}{p^2C_p^2} \|\hat\omega\|_p^p\,,
\end{align*}
where $C_p$ denotes the Poincar\'e constant.
Dividing through by $\|\hat{\omega}\|_p^{p-2}$ we obtain the inequality
\begin{align*}
    \frac{d}{dt}\|\hat\omega\|_p^2 \leq \frac{p-1}{p} \Pra \Ra^2 |\Omega|^\frac{2}{p} - 4 \Pra \frac{p-1}{p^3 C_p^2} \|\hat\omega\|_p^2.
\end{align*}
By the Gr\"onwall inequality
\begin{align*}
    \|\hat\omega\|_p^2 
    &\leq e^{- 4 \Pra \frac{p-1}{p^3 C_p^2} t} \|\hat\omega_0\|_p^2 +\frac{1}{4} C_p^2 p^2 \Ra^2 |\Omega|^\frac{2}{p}
    \leq e^{- 4 \Pra \frac{p-1}{p^3 C_p^2} t}  \|\omega_0-\Lambda\|_p^2 +\frac{1}{4} C_p^2 p^2 \Ra^2 |\Omega|^\frac{2}{p}
    \\
    &
    \leq e^{- 4 \Pra \frac{p-1}{p^3 C_p^2} t}  (\|\omega_0\|_p^2+\Lambda^2|\Omega|^{\frac 2p})+\frac{1}{4} C_p^2 p^2 \Ra^2 |\Omega|^\frac{2}{p}
    \leq C^2(\|\omega_0\|_p^2+\Lambda^2+ \Ra^2).
\end{align*}

\medskip
\textbf{Step 2:}
We now turn to the estimate for $\Lambda$. We have 
\begin{equation*}
    \Lambda = 2\|(\alpha+\kappa)u_\tau\|_{L^\infty(\gamma^-\cup\gamma^+)}
    \leq 2 \|\alpha+\kappa\|_\infty\|u\|_{L^\infty(\Omega\times[0,\bar t])},
\end{equation*}
which can be bounded using interpolation and Young's inequality by
\begin{align*}
    2&\|\alpha+\kappa\|_\infty\|u\|_{L^\infty(\Omega\times[0,\bar t])} 
    \\
    &\qquad\leq C\|\alpha+\kappa\|_\infty\|\nabla u\|_{L^\infty_t (L^p_x)}^\theta \|u\|_{L^\infty_t (L^2_x)}^{1-\theta}+C\|\alpha+\kappa\|_{\infty}\|u\|_{L^\infty_t (L^2_x)}
    \\
    &\qquad\leq \epsilon C \theta \|\nabla u\|_{L^\infty_t (L^p_x)}+ C\left[1+(1-\theta)\epsilon^{\frac{p}{2-p}} \|\alpha+\kappa\|_{\infty}^{\frac{p}{p-2}}  \right] \|\alpha+\kappa\|_{\infty}\|u\|_{L^\infty_t (L^2_x)}
\end{align*}
for $p>2$ and arbitrary $\epsilon>0$, where $\theta= \frac{p}{2(p-1)}$ and $L_t^\infty(L_x^p)=L^\infty([0,\bar t];L^p(\Omega))$. According to Lemma \ref{lemma_u_bounded_by_omega}
\begin{align*}
    \| u \|_{W^{1,p}} &\leq C \left(\|\omega\|_p + \left(1+\|\kappa\|_\infty^{2-\frac{2}{p}}\right)\|u\|_2\right)
\end{align*}
resulting in
\begin{align*}
    \Lambda &\leq \epsilon C \|\omega\|_{L_t^\infty(L_x^p)} + C\left[\epsilon\left(1+\|\kappa\|_\infty^{2-\frac{2}{p}}\right)+\|\alpha+\kappa\|_{\infty}+\epsilon^{\frac{p}{2-p}} \|\alpha+\kappa\|_{\infty}^{\frac{2(p-1)}{p-2}}  \right] \|u\|_{L^\infty_t (L^2_x)}.
\end{align*}

\medskip
\textbf{Step 3:}
Recalling that $|\omega|\leq \max \lbrace |\tilde \omega_-|,|\tilde \omega_+|\rbrace$ and $\hat\omega_+ =\tilde \omega_+ - \Lambda$ and using the results of Step 1 and Step 2 one gets
\begin{align*}
        \|\omega\|_{L_t^\infty (L_x^p)} &\leq\|\tilde \omega\|_{L_t^\infty(L_x^p)} =\|\hat \omega+\Lambda\|_{L_t^\infty(L_x^p)}
        \leq  \|\hat\omega\|_{L_t^\infty(L_x^p)} + |\Omega|^\frac{1}{p} \Lambda
        \leq C(  \|\omega_0\|_p + \Ra +\Lambda)
        \\
        &\leq \epsilon C \|\omega\|_{L_t^\infty ( L_x^p)} + C\|\omega_0\|_p + C \Ra 
        \\
        &\qquad+ C\left[\epsilon\left(1+\|\kappa\|_\infty^{2-\frac{2}{p}}\right)+\|\alpha+\kappa\|_\infty+\epsilon^{\frac{p}{2-p}} \|\alpha+\kappa\|_\infty^{\frac{2(p-1)}{p-2}}\right] \|u\|_{L^\infty_t (L^2_x)}
\end{align*}
    
Choosing $\epsilon$ small we can compensate the vorticity term on the right-hand side. By symmetry of the two boundaries $\max_{\gamma^-\cup\gamma^+} \kappa=-\min_{\gamma^-\cup\gamma^+} \kappa$, which, as $\alpha>0$, implies $\|\kappa\|_\infty\leq \|\alpha+\kappa\|_\infty$. Combining these observations we find
\begin{align*}
    \|\omega\|_{L_t^\infty(L_x^p)} \leq C \|\omega_0\|_p + C\left(1+\|\alpha+\kappa\|_\infty^{\frac{2(p-1)}{p-2}}\right)\|u\|_{L_t^\infty(L_x^2)} + C\Ra.
\end{align*}
Finally Lemma \ref{lemmaEnergyDecay} yields
\begin{align*}
    \|\omega\|_{L_t^\infty(L_x^p)} 
    &\leq C \left[ \|\omega_0\|_p + \left(1+\|\alpha+\kappa\|_\infty^{\frac{2(p-1)}{p-2}}\right)\|u_0\|_{2}+ C_{\alpha,\kappa}\Ra\right],
\end{align*}
where $C_{\alpha,\kappa}=  \left(1+\|\alpha+\kappa\|_\infty^{\frac{2(p-1)}{p-2}}\right)\max\left\lbrace 1, \underline{\alpha}^{-1}\right\rbrace $ and $C$ only depends on $|\Omega|$, $\|h'\|_{\infty}$ and $p$. As the constants are independent of $\bar t$ this bound holds universally in time.

\end{proof}

As $\Omega$ is bounded Hölder inequality and Lemma \ref{LemmaVorticityBound} yield that if $\omega_0 \in L^p$ for any $p<\infty$ then $\|\omega\|_p$ is universally bounded in time. By trace Theorem and Lemma \ref{lemma_u_bounded_by_omega}
\begin{align*}
    \|u\|_{L^2(\gamma^-\cup\gamma^+)}\leq \|u\|_{H^1} \leq C\left(\|\omega\|_2+(1+\|\kappa\|_\infty)\|u\|_2\right).
\end{align*}
Using Lemma \ref{lemmaEnergyDecay} this is also universally bounded in time, therefore taking the long time average of the vorticity balance, see Lemma \ref{vorticity-balance}, we get the following.

\begin{corollary}
Assume that the conditions of Lemma \ref{LemmaVorticityBound} are satisfied and $\omega_0\in L^p$ for some $p>2$, then
\begin{equation}
    \label{average-enstrophy-balance}
    \begin{aligned}
        0 &= \langle |\nabla \omega|^2\rangle - 2 \langle (\alpha+\kappa) u\cdot \nabla p \rangle_{\gamma^-\cup\gamma^+} - \Ra \langle \omega\partial_1 T\rangle 
        \\
        &\qquad + \frac{2}{3\Pra} \langle (\alpha+\kappa) u\cdot (u\cdot\nabla) u\rangle_{\gamma^-\cup\gamma^+} + 2\Ra \langle (\alpha+\kappa) u_\tau n_1\rangle_{\gamma^-}.
    \end{aligned}
\end{equation}
\end{corollary}

\subsection{A-priori estimate for the pressure}
\leavevmode

The pressure satisfies
\begin{equation}
    \label{pressure-eq}
    \begin{aligned}
        \Delta p &= -\frac{1}{\Pra}(\nabla u)^T \colon \nabla u +\Ra \partial_2 T   & \textnormal{ in }&\Omega\\
        n\cdot \nabla p &= -\frac{1}{\Pra}\kappa u_\tau^2 +2 \tau\cdot\nabla\left((\alpha+\kappa)u_\tau\right)  & \textnormal{ on }&\gamma^+\\
        n\cdot \nabla p &= -\frac{1}{\Pra}\kappa u_\tau^2 +2 \tau\cdot\nabla\left((\alpha+\kappa)u_\tau \right) +n_2\Ra  & \textnormal{ on }&\gamma^-
    \end{aligned}
\end{equation}

The equation in the bulk is easy to obtain by applying the divergence to the Navier-Stokes equations, using incompressibility and writing compactly $\nabla\cdot((u\cdot\nabla)u)=(\nabla u)^T:\nabla u$.
In order to track the pressure at the boundary we look at Navier-Stokes equations at $\gminus\cup\gplus$ 
\begin{align*}
    \left(\frac{1}{\Pra}\left( u_t + (u\cdot \nabla)u \right) - \Delta u +  \nabla p - \Ra Te_2\right)\cdot n =0
\end{align*}
where $n$ is the normal at the boundary.
It is clear that
\begin{align*}
    n\cdot u_t &= \frac{d}{dt}(n\cdot u) = 0
\end{align*}
and
\begin{align*}
    n\cdot \Delta u= n \cdot\nabla^\perp \omega =  -2n\cdot \nabla^\perp \left((\alpha+\kappa)u_\tau\right)
    = 2 \tau \cdot \nabla \left((\alpha+\kappa)u_\tau\right),
\end{align*}
using the boundary condition for the vorticity in \eqref{vorticity-equation}.
Thanks to \eqref{appendix-proof-id-kappaUtau2} in the Appendix we also have 
\begin{align}
    \label{id-kappaUtau2-3}
    n\cdot (u\cdot \nabla ) u =\kappa u_\tau^2.
\end{align}
Hence
\begin{align*}
    n\cdot \nabla p = -\frac{\kappa}{\Pra} u_\tau^2 +2\tau \cdot \nabla ((\alpha+\kappa)u_\tau) + n_2T\Ra
\end{align*}
at the boundary. Now, it is only left to observe that $T=0$ at $\gplus$ and $T=1$ at $\gminus$.

\begin{proposition}
\label{proposition-pressure-bound}
For any $r\in (2,\infty)$ there exists a constant $C$ depending on $|\Omega|$, $r$ and $\|h'\|_{\infty}$ such that
\begin{align*}
    \| p\|_{H^1}
    \leq C\left[ \Ra\|T\|_2 + \|\alpha+\kappa\|_\infty\|u\|_{H^2}+\left(\frac{1+\|\kappa\|_\infty}{\Pra}\|u\|_{W^{1,r}}+\|\dot\alpha+\dot\kappa\|_\infty \right)\|u\|_{H^1}\right].
\end{align*}

\end{proposition}

\begin{proof}
On one hand, integrating by parts and using the boundary conditions (for $u$ and $T$), we have 
\begin{align*}
    \int_{\Omega} p\Delta p 
    &= - \| \nabla p\|_2^2 - \frac{1}{\Pra} \int_{\gamma^-\cup\gamma^+} p\kappa u_\tau^2 +2\int_{\gamma^-\cup\gamma^+}p \tau\cdot \nabla\left((\alpha+\kappa)u_\tau\right)+\Ra\int_{\gamma^-} p n_2
\end{align*}
On the other hand using the equation satisfied by the pressure \eqref{pressure-eq}
\begin{align*}
    \int_{\Omega} p\Delta p 
     &= -\frac{1}{\Pra}\int_{\Omega} p(\nabla u)^T \colon \nabla u + \Ra \int_{\Omega} p\partial_2 T \\
     &= -\frac{1}{\Pra}\int_{\Omega} p(\nabla u)^T \colon \nabla u  + \Ra \int_{\gminus\cup\gplus} pT n_2- \Ra \int_{\Omega} T\partial_2 p\\
     &= -\frac{1}{\Pra}\int_{\Omega} p(\nabla u)^T \colon \nabla u + \Ra\int_{\gminus} pn_2- \Ra \int_{\Omega} T\partial_2 p,
\end{align*}
where we used the boundary conditions for $T$ in the last identity.
Combining these estimates one gets
\begin{align*}
     \| \nabla p\|^2 
     &= - \frac{1}{\Pra} \int_{\gamma^-\cup\gamma^+} p\kappa u_\tau^2 +2\int_{\gamma^-\cup\gamma^+}p \tau\cdot \nabla\left((\alpha+\kappa)u_\tau\right) +\frac{1}{\Pra}\int_{\Omega} p(\nabla u)^T \colon \nabla u
     \\&\qquad + \Ra \int_{\Omega} T\partial_2 p.
\end{align*}
We estimate the right-hand side:
By H\"older inequality with $\frac{1}{r}+\frac{1}{q}+\frac{1}{2}=1$ and Sobolev embedding\footnote{ Since $\Omega$ is bounded, for any $1\leq \mu<\infty$ choose $s>\max\lbrace 2,\mu \rbrace$ and $q=\frac{s}{\mu}>1$, such that $\frac{ns}{n+s}=\frac{2s}{2+s}\leq \frac{2s}{s}=2$. By H\"older and Sobolev inequality we obtain
\begin{align*}
    \|f\|_\mu \leq \|1\|_{\frac{\mu q}{q-1}}\|f\|_{\mu q} = |\Omega|^{\frac{q-1}{\mu q}}\|f\|_s =|\Omega|^{\frac{\mu s-1}{s}}\|f\|_s \leq C \|f\|_{W^{1,\frac{ns}{n+s}}} \leq C \|f\|_{H^1}.
\end{align*}
}
\begin{align}
    \label{pressure-3-term-estimate}
    \|p(\nabla u)^T\colon \nabla u\|_1 &\leq \|p\|_q \|\nabla u\|_2 \|\nabla u\|_r 
    \leq C \|p\|_{H^1} \|\nabla u\|_2 \|\nabla u\|_r,
\end{align}
where $C$ depends on $|\Omega|$, $r$ and $\|h'\|_{\infty}$. For the temperature term we apply H\"older's inequality use that $\|T\|_\infty\leq 1$ by the maximum principle \eqref{maximum-principle}
\begin{align*}
    \|T\partial_2 p\|_1\leq \|T\|_2 \|p\|_{H^1}\leq |\Omega|^\frac{1}{2} \|p\|_{H^1}\,,
\end{align*}
and for the second term we compute
\begin{align*}
    &\int_{\gamma^-\cup\gamma^+} \left|p \tau\cdot \nabla ((\alpha+\kappa)u_{\tau})\right|
    \\
    &\quad\leq \|\alpha+\kappa\|_\infty\|p\|_{L^2(\gamma^-\cup\gamma^+)} \|\nabla u\|_{L^2(\gamma^-\cup\gamma^+)} + \|\dot\alpha+\dot\kappa\|_\infty \|p\|_{L^2(\gamma^-\cup\gamma^+)}\|u\|_{L^2(\gamma^-\cup\gamma^+)}
    \\
    &\quad\leq C \|\alpha+\kappa\|_\infty \|p\|_{H^1} \|u\|_{H^2} + C \|\dot\alpha+\dot\kappa\|_\infty \|p\|_{H^1} \|u\|_{H^1},
\end{align*}
where in the last inequality we use the trace estimate.
Finally, we estimate the first term: Similar to \eqref{pressure-3-term-estimate} for $\frac{1}{r}+\frac{1}{q}+\frac{1}{2}=1$ and every $r>2$
\begin{align*}
    \left|\int_{\gamma^-\cup \gamma^+} \kappa p u_\tau^2\right| &\leq \|\kappa\|_\infty \|pu_\tau^2\|_{L^1(\gamma^-\cup\gamma^+)}\\
    &\leq C\|\kappa\|_\infty  \|pu^2\|_{W^{1,1}} \\
    &\leq C\|\kappa\|_\infty \left( \|p\|_q \|u\|_2\|u\|_r + \|\nabla p\|_2 \|u\|_r \|u\|_q + \| p \|_q \|u\|_r \|\nabla u\|_2\right)\\
    &\leq C\|\kappa\|_\infty \left( \|p\|_{H^1} \|u\|_{H^1}\| u\|_{W^{1,r}} + \|\nabla p\|_2 \|u\|_r \|u\|_{H^1} + \| p \|_{H^1} \|u\|_r \|\nabla u\|_2\right)\\
    &\leq C \|\kappa\|_\infty \|p\|_{H^1} \|u\|_{H^1}\| u\|_{W^{1,r}}\,,
\end{align*}
where $C$ depends on $|\Omega|$, $r$ and $\|h'\|_{\infty}$.

\smallskip 

Combining the estimates we find
\begin{align*}
     \| \nabla p\|_2^2 
     &\leq C \|p\|_{H^1}\left[ \Ra + \|\alpha+\kappa\|_\infty\|u\|_{H^2}+\left(\frac{1+\|\kappa\|_\infty}{\Pra}\|u\|_{W^{1,r}}+\|\dot\alpha+\dot\kappa\|_\infty \right)\|u\|_{H^1}\right].
\end{align*}

Using that the pressure $p$ is only defined up to a constant so we choose $p$ to have zero mean such that Poincaré yields $\|p\|_q\leq C \|\nabla p\|_q$ which implies $\|p\|_{H^1}^2 = \|\nabla p\|_2^2+\|p\|_2^2 \leq (1+C^2) \|\nabla p\|_2^2$. Then 
\begin{align*}
    \| p\|_{H^1}^2
    &\leq C \| p\|_{H^1}\left[ \Ra + \|\alpha+\kappa\|_\infty\|u\|_{H^2}+\left(\frac{1+\|\kappa\|_\infty}{\Pra}\|u\|_{W^{1,r}}+\|\dot\alpha+\dot\kappa\|_\infty \right)\|u\|_{H^1}\right]\,.
\end{align*}
Finally dividing by $\| p\|_{H^1}$ we conclude that there exists a constant $C>0$ depending on $\Omega$ and $r$ such that
\begin{eqnarray*}
    \| p\|_{H^1}
    \leq C\left[ \Ra + \|\alpha+\kappa\|_\infty\|u\|_{H^2}+\left(\frac{1+\|\kappa\|_\infty}{\Pra}\|u\|_{W^{1,r}}+\|\dot\alpha+\dot\kappa\|_\infty \right)\|u\|_{H^1}\right]\,.
\end{eqnarray*}
for any $r>2$.
\end{proof}

\section{Upper bounds on the Nusselt number}\label{section-four}
Combining the a-priori estimates derived in the previous section we are now able to prove the $\Ra^\frac{1}{2}$ bound, that was first derived for the flat, no slip case in 3 dimensions by Doering and Constantin \cite{DC96}.
\subsection{Proof of Theorem \ref{Lemma-Ra-One-Half-Bound}}

Let $\Omega_\delta$ be given by $$\Omega_\delta = \lbrace (y_1,y_2) \ \vert \ 0\leq y_1\leq \Gamma, 1+h(y_1)-\delta\leq y_2\leq 1+h(y_1)\rbrace$$ as illustrated in Figure \ref{fig:omega_delta}.
\begin{figure}
    \begin{center}
        \includegraphics[width=0.5\textwidth]{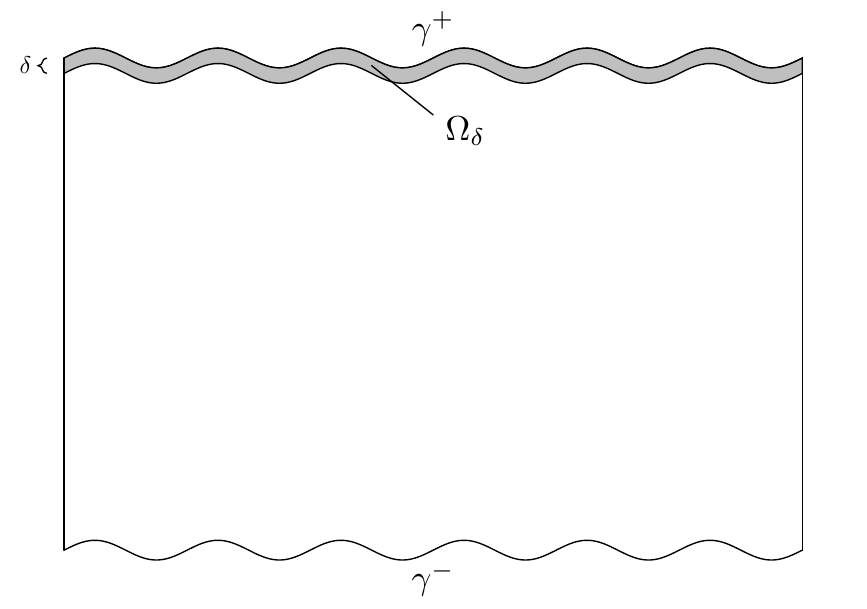}
    \end{center}
    \caption{Illustration of $\Omega_\delta$.}
    \label{fig:omega_delta}
\end{figure}
Taking the average in $z\in (1-\delta,1)$ in the representation of the Nusselt number \eqref{nusselt-strip} we find
\begin{equation}
    \label{halfBound-nusselt-on-Omega_delta}
    \begin{aligned}
        \Nu &= \limsup_{T\to\infty} \frac{1}{T}\int_0^T \frac{1}{\delta} \frac{1}{|\Omega|} \int_{1-\delta}^{1}\int_{\gamma(z)} n_+\cdot (u-\nabla)T \ dS\ dz\ dt 
        \\
        &= \limsup_{T\to\infty} \frac{1}{T}\int_0^T \frac{1}{\delta} \frac{1}{|\Omega|} \int_{\Omega_\delta} n_+\cdot u T \ dy\ dt - \limsup_{T\to\infty} \frac{1}{T}\int_0^T \frac{1}{\delta} \frac{1}{|\Omega|} \int_{\Omega_\delta} n_+\cdot \nabla T \ dy\ dt
    \end{aligned}
\end{equation}
In order to estimate the first term on the right-hand side notice that by the fundamental theorem of calculus for $(y_1,y_2)\in \Omega_\delta$
\begin{equation}
    \label{halfBound-velocity-term}
    \begin{aligned}
        |n_+\cdot u| (y_1,y_2)&= \left|n_+\cdot u\vert_{\gamma^+}+\int_{1+h(y_1)}^{y_2} \partial_2 (n_+\cdot u)\ dz\right| \leq  \int_{1+h(y_1)-\delta}^{1+h(y_1)} |\partial_2 u|\ dz 
        \\
        &\leq \delta^\frac{1}{2}\|\nabla u\|_{L^2(\gamma^-,\gamma^+)},
    \end{aligned}
\end{equation}
where $\| \nabla u\|_{L^2(\gamma^-,\gamma^+)}=\|\nabla u(y_1,\cdot)\|_{L^2(h(y_1),1+h(y_1))}$ and we used the non-penetration boundary condition for $u$ and that $n_+$ is constant in $y_2$-direction in the first inequality and Hölder's inequality in the second estimate. Analogously for the temperature and $(y_1,y_2)\in \Omega_\delta$ it holds
\begin{align}
    \label{halfBound-temperature-term}
    |T|(y_1,y_2) \leq \delta^\frac{1}{2} \| \nabla T\|_{L^2(\gamma^-,\gamma^+)}
\end{align}
as $T=0$ on $\gamma^+$. In order to estimate the second integral in \eqref{halfBound-nusselt-on-Omega_delta} partial integration and the boundary condition $T=0$ on $\gamma^+$ yields
\begin{align}
    \label{halfBound-nusselt-second-integral-estimate-1}
    \left|\int_{\Omega_\delta} n_+\cdot \nabla T \ dy \right| \leq \int_{\gamma^+} |T| \ dS+ \int_{\gamma(1-\delta)} |n_+\cdot n_- T| \ dS + \int_{\Omega_\delta} |T\nabla \cdot n_+ | \ dy.
\end{align}
By the maximum principle \eqref{maximum-principle} the temperature is bounded by $\|T\|_\infty \leq 1$, so the first two terms on the right-hand side of \eqref{halfBound-nusselt-second-integral-estimate-1} are bounded by a constant depending on $\|h'\|_{\infty}$ and $|\Omega|$. In order to estimate the last term notice that
\begin{align*}
    n_+=\frac{1}{\sqrt{1+(h')^2}}\begin{pmatrix}-h'\\1\end{pmatrix}\qquad \textnormal{ and }\qquad |\kappa| = \frac{|h''|}{(1+(h')^2)^\frac{3}{2}},
\end{align*}
where $h'=\partial_1 h(y_1)$ and $h''= \partial_1^2h(y_1)$ as derived in \eqref{appendix-normal-vector-representation} and \eqref{appendix-kappa-representation} in the Appendix. Therefore $|\nabla\cdot n_+ |=|\kappa|$, which implies
\begin{align*}
    \int_{\Omega_\delta} |T\nabla \cdot n_+ | \ dy \leq \|\kappa\|{_\infty}|\Omega_\delta|
\end{align*}
for the last term in \eqref{halfBound-nusselt-second-integral-estimate-1}.
Combining these observations
\begin{align}
    \label{halfBound-nusselt-second-integral-estimate-2}
    \left|\int_{\Omega_\delta} n_+\cdot \nabla T \ dy \right| \leq C + \delta\Gamma \|\kappa\|_\infty.
\end{align}
Plugging \eqref{halfBound-velocity-term}, \eqref{halfBound-temperature-term} and \eqref{halfBound-nusselt-second-integral-estimate-2} into \eqref{halfBound-nusselt-on-Omega_delta} and using Hölder inequality, there exists a constant depending on $\|h'\|_{\infty}$ and $|\Omega|$ such that
\begin{align*}
    \Nu &\leq
    \limsup_{T\to\infty} \frac{1}{T}\int_0^T \frac{1}{|\Omega|} \int_{\Omega_\delta} \|\nabla u\|_{L^2(\gamma^-,\gamma^+)}\|\nabla T\|_{L^2(\gamma^-,\gamma^+)} \ dy\ dt + C\frac{1}{\delta}+\|\kappa\|_\infty
    \\
    &\leq C \left(\delta \langle |\nabla u|^2 \rangle^\frac{1}{2} \langle |\nabla T|^2 \rangle^\frac{1}{2}+\frac{1}{\delta}\right)+\|\kappa\|_\infty.
\end{align*}
By \eqref{nusselt-gradT} and \eqref{average-energy-balance} we can substitute both gradients and get
\begin{align*}
    \Nu &\leq C \left(\delta\Ra^\frac{1}{2}\left((1+\max h-\min h) \Nu-1\right)^\frac{1}{2} \Nu^\frac{1}{2}+\frac{1}{\delta}\right)+\|\kappa\|_\infty 
    \\
    &\leq  C\left(\delta\Ra^\frac{1}{2} \Nu+\frac{1}{\delta}\right)+\|\kappa\|_\infty.
\end{align*}
Balancing the terms by choosing $\delta=\Nu^{-\frac{1}{2}}\Ra^{-\frac{1}{4}}$ we get
\begin{align*}
    \Nu \leq  C\Ra^\frac{1}{2}+2\|\kappa\|_\infty
\end{align*}
for $\Ra \geq 1$.

\begin{remark}\label{noslip-remark}
We notice that the same proof (with minor modifications) would yield $\Nu\lesssim \Ra^{\frac 12}+\|\kappa\|_{\infty}$, for a flow with no-slip boundary conditions, which is the case considered by Goluskin and Doering. For comparison, in \cite{GD16} the authors proved $\Nu\lesssim \Ra^{\frac 12}$, where the constant only depends on $\|\nabla h\|_2$, by using the background field method.
\end{remark}
\subsection{Introduction of the background field method}
\leavevmode

In order to improve the bound of Theorem \ref{Lemma-Ra-One-Half-Bound} we follow the "background field" strategy used \cite{drivasNguyenNobiliBoundsOnHeatFluxForRayleighBenardConvectionBetweenNavierSlipFixedTemperatureBoundaries}, which is based on \cite{whiteheadDoeringUltimateState}. This approach consists of specifying a stationary background field for the temperature and show its "marginal stability" as we will explain in what follows. This will be achieved by applying the a-priori bounds derived in Section \ref{Section-A-Priori-Bounds}.

To this end we define the background profile for the temperature by
\begin{align}
    \label{def-eta}
    \eta(y_1,y_2)=1-\frac{1}{2\delta}
    \begin{cases}
        \begin{aligned}
            &2\delta+y_2-(1+h(y_1)) & &\textnormal{for}& 1+h(y_1)-\delta&\leq y_2 \leq 1+h(y_1)\\
            &\delta & &\textnormal{for}& h(y_1)+\delta&< y_2 < 1+h(y_1)-\delta\\
            &y_2-h(y_1) & &\textnormal{for}& h(y_1)&\leq y_2 \leq h(y_1)+\delta
        \end{aligned}
    \end{cases}
\end{align}
for $\delta>0$ and the difference $\theta$ by
\begin{align}
    \label{def-theta}
    \theta = T-\eta.
\end{align}
This profile is illustrated in Figure \ref{fig:background_profile}.

\begin{figure}
    \begin{center}
        \includegraphics[width=0.5\textwidth]{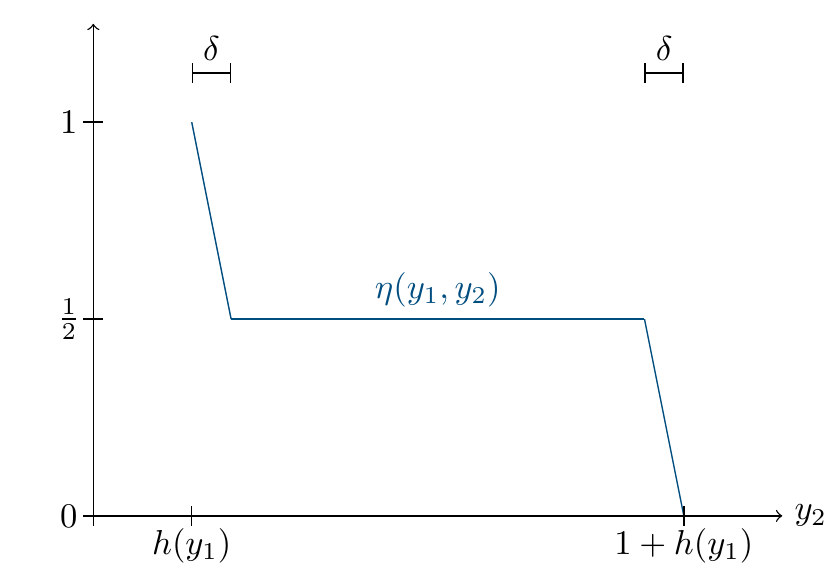}
    \end{center}
    \caption{Illustrations of the background profile $\eta$.}
    \label{fig:background_profile}
\end{figure}

Note that $\eta$ fulfills the boundary conditions of $T$, so $\theta$ vanishes on $\gamma^\pm$. Also since $n_-$ can be expressed by
\begin{align*}
    n_- = \frac{1}{\sqrt{1+(h')^2}}\begin{pmatrix}h'\\ -1\end{pmatrix}
\end{align*}
as derived in \eqref{appendix-normal-vector-representation}, its gradient is given by
\begin{equation}
    \label{nabla-eta-identity}
    \nabla \eta = 
    \begin{cases}
        \begin{aligned}
            &0 &&\textnormal{for}& h(y_1)+\delta < y_2 < 1+h(y_1)-\delta
            &\frac{1}{2\delta} \sqrt{1+(h')^2}n_- &&\textnormal{else}
        \end{aligned}
    \end{cases}
\end{equation}
for almost every $y\in \Omega$. Inserting this decomposition in the definition of the Nusselt number, we have the following.

\begin{proposition}
Let $\eta$ and $\theta$ be defined by \eqref{def-eta} and \eqref{def-theta}. Then
\begin{align}
    \label{nusselt-eta-theta-representation}
    \Nu = \langle |\nabla \eta|^2\rangle -\langle |\nabla\theta|^2\rangle - 2\langle \theta u\cdot \nabla \eta\rangle
\end{align}
\end{proposition}

The proof of this identity is essentially the same as the one for Proposition 7 in \cite{drivasNguyenNobiliBoundsOnHeatFluxForRayleighBenardConvectionBetweenNavierSlipFixedTemperatureBoundaries} and we report it here just for convenience of the reader.
\begin{proof}
Plugging the definitions of $\eta$ and $\theta$ into \eqref{heatEquation} we have
\begin{align*}
    \theta_t + u\cdot\nabla \eta +u\cdot \nabla \theta -\Delta \eta -\Delta \theta = 0,
\end{align*}
and, integrating this equation against $\theta$ we find
\begin{align}
    \label{theta-energy-identity}
    0=\frac{1}{2}\frac{d}{dt}\|\theta\|_2^2 + \int_\Omega \theta u\cdot \nabla \eta + \int_\Omega \theta u\cdot \nabla \theta -\int_\Omega \theta\Delta \eta-\int_\Omega \theta\Delta\theta.
\end{align}
The third term on the right-hand side of \eqref{theta-energy-identity} vanishes since $u\cdot n=0$ and $u$ is divergence-free. For the fourth term on the right-hand side of \eqref{theta-energy-identity} we get
\begin{align*}
    -\int_\Omega \theta \Delta \eta  &= - \int_{\gamma^-\cup\gamma^+} \theta n\cdot \nabla \eta  + \int_\Omega \nabla\theta \cdot\nabla \eta  = \int_\Omega \nabla\theta \cdot\nabla \eta ,
\end{align*}
where in the last equality we used that $\theta$ vanishes on the boundary by definition. Similarly
\begin{align*}
    -\int_\Omega \theta \Delta \theta = \|\nabla\theta\|_2^2.
\end{align*}
Therefore taking the long time average of \eqref{theta-energy-identity} and using that $\theta$ is universally bounded in time as both $T$ and $\eta$ fulfill $0\leq T,\eta\leq 1$ we find
\begin{align*}
    \langle \theta u\cdot\nabla \eta \rangle + \langle \nabla\theta \cdot \nabla \eta \rangle + \langle |\nabla\theta|^2\rangle = 0.
\end{align*}
Using this identity in \eqref{nusselt-gradT} we find
\begin{align*}
    \Nu = \langle|\nabla T|^2 \rangle = \langle |\nabla \eta|^2\rangle + 2\langle \nabla \theta \cdot\nabla\eta \rangle + \langle |\nabla \theta|^2\rangle = \langle |\nabla \eta|^2\rangle - \langle |\nabla \theta|^2\rangle -2\langle \theta u \cdot\nabla\eta \rangle
\end{align*}
as a representation of $\Nu$.\footnote{The argument can be rigorously justified via mollification of $\eta$.}
\end{proof}

Next we define
\begin{equation*}
    \label{def-a}
    \begin{aligned}
        \mathbf{a} &:= \langle |\nabla \omega|^2\rangle -2 \langle (\alpha+\kappa)u\cdot\nabla p\rangle_{\gamma^-\cup\gamma^+} - \Ra \langle \omega\partial_1 T\rangle 
        \\
        &\qquad + \frac{2}{3\Pra} \langle (\alpha+\kappa) u\cdot (u\cdot\nabla) u\rangle_{\gamma^-\cup\gamma^+} + 2\Ra \langle (\alpha+\kappa) u_\tau n_1\rangle_{\gamma^-} =0,
    \end{aligned}
\end{equation*}
where the last identity is due to \eqref{average-enstrophy-balance} and 
\begin{align}
    \label{def-b}
    \mathbf{b} &:= \langle|\nabla u|^2\rangle + \langle (2\alpha+\kappa) u_\tau^2\rangle_{\gamma^-\cup \gamma^+} - \Ra\left((1+\max h - \min h) \Nu-1\right)\,.
\end{align}
Using \eqref{nusselt-eta-theta-representation} we can rewrite the Nusselt number as
\begin{align}
    \label{nu-identity-including-Q}
    (1-b(1+&\max h -\min h))\Nu + b 
    = M\Ra^2+2\langle |\nabla\eta|^2\rangle-\mathcal{Q}[u,\theta,\eta]
\end{align}
where the quadratic form $\mathcal{Q}$ is defined as
\begin{multline}\label{quadratic-form}
    \mathcal{Q}[u,\theta,\eta]:=M \Ra^2 + \langle |\nabla \eta|^2\rangle + \langle |\nabla \theta|^2\rangle + 2\langle \theta u\cdot\nabla \eta\rangle 
    \\
    \qquad + \frac{b}{\Ra}\langle |\nabla u|^2\rangle + \frac{b}{\Ra}\langle (2\alpha+\kappa)u_\tau^2\rangle_{\gamma^-\cup\gamma^+} - \frac{b}{\Ra}\mathbf{b} + a \mathbf{a}\,.
\end{multline}
In this new representation $a>0$ and $0\leq b<(1+\max h - \min h)^{-1}$ and notice that the balancing term $M\Ra^2$, with $M>0$, was introduced. The choice of the parameters $a,b,M$ will follow from an optimization procedure at the end.

We now want to prove that for a suitable choice of $\delta$ the form $\mathcal{Q}$ is non negative. In order to do so we need the following Lemma.
\begin{lemma}
\label{lemma-estimate-theta-u-gradeta}
One has
\begin{align*}
    2|\langle\theta u\cdot \nabla \eta\rangle| &\leq  \delta^6 C (a\epsilon)^{-1}\langle |\partial_2 u|^2\rangle + a\epsilon \langle |\partial_2^2 u|^2\rangle +\frac{1}{2}\langle |\partial_2 \theta|^2\rangle
\end{align*}
for any $\epsilon>0$.
\end{lemma}

\begin{proof}
By \eqref{nabla-eta-identity}
\begin{equation}
    \label{theta-u-grad-eta-estimate-1}
    \begin{aligned}
        2\int_\Omega \theta u\cdot\nabla \eta &= \frac{1}{\delta} \int_0^\Gamma \int_{h(y_1)}^{h(y_1)+\delta} \sqrt{1+(h')^2} \theta u\cdot n_-\ dy_2\ dy_1
        \\
        &\qquad + \frac{1}{\delta} \int_0^\Gamma\int_{1+h(y_1)-\delta}^{1+h(y_1)} \sqrt{1+(h')^2} \theta u\cdot n_-\ dy_2 \ dy_1
    \end{aligned}
\end{equation}
We focus on the first term on the right-hand side. The second one can be treated similarly. By the fundamental theorem of calculus and Hölder's inequality
\begin{equation}
    \label{theta-u-grad-eta-estimate-u}
    \begin{aligned}
        |u\cdot n_-|(y_1,y_2)&=\left|(u\cdot n_-)(y_1,h(y_1))+ \int_{h(y_1)}^{y_2} \partial_2(u\cdot n_-)(y_1,z) \ dz\right| 
        \\
        &\leq \delta \|\partial_2(u\cdot n_-)\|_{L^\infty(\gamma^-,\gamma^+)}
    \end{aligned}
\end{equation}
for $h(y_1)\leq y_2\leq h(y_1)+\delta$, where in the last inequality we used the boundary condition for $u$. Similarly
\begin{align}
    \label{theta-u-grad-eta-estimate-theta}
    |\theta(y_1,y_2)|\leq \delta^\frac{1}{2}\|\partial_2 \theta\|_ {L^2(\gamma^-,\gamma^+)}
\end{align}
as $\theta$ vanishes on the boundary. In order to estimate $\partial_2 (n_-\cdot u)$ notice that by partial integration and the boundary condition for $u$
\begin{align}
    \label{different_boundaries_not_0_anymore}
    \int_{h(y_1)}^{1+h(y_1)} \partial_2 (n_- \cdot u) = n_2 \ n_-\cdot u\vert_{\gamma^-} - n_2 \ n_-\cdot u\vert_{\gamma^+} = 0.
\end{align}
Therefore for every $y_1$ there exists $h(y_1)\leq \bar y_2\leq 1+h(y_1)$ such that $\partial_2(u\cdot n_-)(y_1,\bar y_2)=0$. Applying the fundamental theorem of calculus again we find
\begin{equation}
    \label{theta-u-grad-eta-estimate-d22u}
    \begin{aligned}
        (\partial_2 (u\cdot n_-))^2(y_1,y_2)&= (\partial_2 (u\cdot n_-))^2 (y_1,\bar y_2) + \int_{\bar y_2}^{y_2} \partial_2 \left((\partial_2(u\cdot n_-))^2\right)(y_1,z)\ dz
        \\
        &\leq 2 \|\partial_2 u\|_{L^2(\gamma^-,\gamma^+)}\|\partial_2^2 u\|_{L^2(\gamma^-,\gamma^+)},
    \end{aligned}
\end{equation}
where in the last inequality we used Hölder's inequality, that $n_-$ is constant in $e_2$ direction and $|n_-|=1$. Combing \eqref{theta-u-grad-eta-estimate-1} with \eqref{theta-u-grad-eta-estimate-u}, \eqref{theta-u-grad-eta-estimate-theta} and \eqref{theta-u-grad-eta-estimate-d22u} and using Young's inequality twice yields
\begin{align*}
    \bigg|2\int_\Omega &\theta u\cdot\nabla \eta\ dy\bigg|
    \\
    &\leq (2 \delta)^\frac{3}{2}\int_0^\Gamma \|\partial_2 u\|_{L^2(\gamma^-,\gamma^+)}^\frac{1}{2}\|\partial_2^2 u\|_{L^2(\gamma^-,\gamma^+)}^\frac{1}{2} \|\partial_2\theta\|_{L^2(\gamma^-,\gamma^+)}\sqrt{1+(h')^2}\ dy_1
    \\
    &\leq  C\delta^\frac{3}{2} \int_0^\Gamma\mu\nu\|\partial_2 u\|_{L^2(\gamma^-,\gamma^+)}^2+\mu\nu^{-1}\|\partial_2^2 u\|_{L^2(\gamma^-,\gamma^+)}^2 + \mu^{-1} \|\partial_2\theta\|_{L^2(\gamma^-,\gamma^+)}^2\ dy_1
\end{align*}
for some $\mu,\nu>0$ that will be determined later and $C=\left\|\sqrt{1+(h')^2}\right\|_{\infty}$.

Taking the long time average
\begin{align*}
    2|\langle\theta u\cdot \nabla \eta\rangle| \leq C\delta^\frac{3}{2} \left(\mu\nu \langle |\partial_2 u|^2\rangle + \mu \nu^{-1} \langle |\partial_2^2 u|^2\rangle +\mu^{-1}\langle |\partial_2 \theta|^2\rangle \right)
\end{align*}
and setting $\mu = 2\delta^\frac{3}{2}C$ and $\nu=2\delta^3 C^2 (a\epsilon)^{-1}$ yields the result.
\end{proof}

With all these preparations at hand we are able to prove the main result in the next subsection.
\subsection{Proof of Theorem \ref{main-theorem}}
In the following we will extensively use
\begin{align}
    \label{proof-main-theorem-essinf-alpha-and-kappa-leq-1}
    \underline{\alpha}\leq 1, \qquad \|\kappa\|_\infty\leq 1.
\end{align}
The first inequality is justified as $\essinf \alpha \leq \essinf_{\kappa>0} (\alpha+\kappa) \leq \|\alpha+\kappa\|_\infty \leq 1$ by assumption \eqref{theorem-condition-alpha+kappa-small} and the second one as $\kappa(y_1,h(y_1))=-\kappa(y_1,1+h(y_1))$ and $\alpha>0$ almost everywhere one has $-\essinf_{\kappa <0}\kappa=\esssup_{\kappa>0}\kappa\leq \esssup_{\kappa>0} \alpha+\kappa\leq \|\alpha+\kappa\|_\infty\leq 1$ by assumption \eqref{theorem-condition-alpha+kappa-small}.

We will show that $\mathcal{Q}$ is non-negative for some appropriate choice of $\delta$. Then \eqref{nu-identity-including-Q} will yield the bound.

As $\mathbf{b}\leq 0$ by \eqref{average-energy-balance} plugging in the definition of $\mathbf{a}$, i.e. \eqref{def-a}, yields
\begin{equation}
    \label{Q-estimation-1}
    \begin{aligned}
        \mathcal{Q}[u,\theta,\eta]&=M \Ra^2 + \langle |\nabla\eta|^2\rangle + \langle |\nabla \theta|^2\rangle + 2\langle \theta u\cdot\nabla \eta\rangle + \frac{b}{\Ra}\langle |\nabla u|^2\rangle + \frac{b}{\Ra}\langle (2\alpha+\kappa)u_\tau^2\rangle_{\gamma^-\cup\gamma^+}
        \\
        &\qquad - \frac{b}{\Ra}\mathbf{b} + a \mathbf{a}
        \\
        &\geq M \Ra^2 + \langle |\nabla\eta|^2\rangle + \langle |\nabla \theta|^2\rangle + 2\langle \theta u\cdot\nabla \eta\rangle + \frac{b}{\Ra}\langle |\nabla u|^2\rangle + \frac{b}{\Ra}\langle (2\alpha+\kappa)u_\tau^2\rangle_{\gamma^-\cup\gamma^+} 
        \\
        &\qquad + a\langle |\nabla \omega|^2\rangle - 2 a \langle (\alpha+\kappa)u\cdot\nabla p\rangle_{\gamma^-\cup\gamma^+} - a \Ra \langle \omega\partial_1 T\rangle 
        \\
        &\qquad + \frac{2}{3\Pra} a \langle (\alpha+\kappa) u\cdot (u\cdot\nabla) u\rangle_{\gamma^-\cup\gamma^+} + 2 a \Ra \langle (\alpha+\kappa) u_\tau n_1\rangle_{\gamma^-}.
    \end{aligned}
\end{equation}
Next we estimate some of the terms individually. 
\begin{itemize}
    \item
    For the eighth term on the right-hand side of \eqref{Q-estimation-1} we can shift the derivative onto $u$ and $\alpha+\kappa$ as the boundary is periodic and get
    \begin{align*}
        -\int_{\gamma^-\cup\gamma^+} (\alpha+\kappa)u\cdot \nabla p \ dS &= \langle p \tau \cdot \nabla ((\alpha+\kappa) u_\tau)\rangle_{\gamma^-\cup\gamma^+} 
        \\
        &= \langle (\alpha+\kappa) p \tau \cdot \nabla u_\tau\rangle_{\gamma^-\cup\gamma^+} + \langle p (\dot\alpha+\dot\kappa) u_\tau\rangle_{\gamma^-\cup\gamma^+} 
    \end{align*}
    Using Hölder's inequality and Trace Theorem one gets
    \begin{align*}
        \left|\int_{\gamma^-\cup\gamma^+} (\alpha+\kappa)u\cdot \nabla p \ dS\right| \leq C \left( \|\alpha+\kappa\|_\infty \|u\|_{H^2} + \|\dot\alpha+\dot\kappa\|_\infty\|u\|_{H^1} \right)\|p\|_{H^1}
    \end{align*}
    where $\dot \alpha$ and $\dot \kappa$ denotes the derivative of $\alpha$ and $\kappa$ along the boundary. The pressure bound derived in Proposition \ref{proposition-pressure-bound} and Young's inequality imply
    \begin{equation}
        \label{Q-estimation-up}
        \begin{aligned}
            &2\left|\int_{\gamma^-\cup\gamma^+} (\alpha+\kappa)u\cdot \nabla p  \ dS\right|
            \\
            &\qquad \leq C \left( \|\alpha+\kappa\|_\infty \|u\|_{H^2} + \|\dot\alpha+\dot\kappa\|_\infty\|u\|_{H^1} \right)
            \\
            &\qquad\qquad\cdot\left[Ra\|T\|_2 + \|\alpha+\kappa\|_\infty\|u\|_{H^2}+\left(\frac{1+\|\kappa\|_\infty}{\Pra}\|u\|_{W^{1,r}}+\|\dot\alpha+\dot\kappa\|_\infty \right)\|u\|_{H^1}\right]
            \\
            &\qquad \leq \left(\epsilon+C\|\alpha+\kappa\|_\infty^2\right) \|u\|_{H^2}^2 + C_\epsilon \|\alpha+\kappa\|_{W^{1,\infty}}^2 \Ra^2  
            \\
            &\qquad\qquad + C\left( \left(\frac{1}{\Pra}\|u\|_{W^{1,r}}\right)^2+\|\alpha+\kappa\|_{W^{1,\infty}}^2+1 \right)\|u\|_{H^1}^2
        \end{aligned}
    \end{equation}
    for all $\epsilon>0$, where $C_\epsilon>0$ depends on $|\Omega|$, $r$, $\|h'\|_{\infty}$ and $\epsilon$ and in the last inequality we used that $\|\kappa\|\leq 1$.
    \item
    For the ninth term on the right-hand side of \eqref{Q-estimation-1} Hölder's and Young's inequality yield
    \begin{align*}
        |a\Ra \langle \omega\partial_1 T\rangle| &= |a\Ra \langle \omega\partial_1 (\eta+\theta)\rangle| \leq \frac{1}{2} \langle|\nabla \eta|^2\rangle + \frac{1}{2} \langle|\nabla \theta|^2\rangle + a^2 \Ra^2 \langle |\omega|^2\rangle.
    \end{align*}
    \item
    In order to estimate the tenth term on the right-hand side of \eqref{Q-estimation-1} we first use Hölder's inequality and Trace Theorem to get
    \begin{align}
        \label{Q-estimation-uuu-1}
        \frac{1}{\Pra}&\left|\int_{\gamma^-\cup\gamma^+} (\alpha+\kappa) u\cdot(u\cdot \nabla)u \ dS\right| 
        \leq C\frac{\|\alpha+\kappa\|_\infty}{\Pra} \left\|u^2 |\nabla u|\right\|_{W^{1,1}} .
    \end{align}
    Again Hölder's inequality with $\frac{1}{r}+\frac{1}{p}+\frac{1}{2}=1$ and Sobolev Theorem as in the proof of Proposition \ref{proposition-pressure-bound} imply
    \begin{equation}
        \label{Q-estimation-uuu-2}
        \begin{aligned}
            \left\|u^2 |\nabla u|\right\|_{W^{1,1}} &\leq C\left(\|u\|_q\|u\|_r \|\nabla u\|_2 + \|u\|_q\|\nabla u\|_r \|\nabla u\|_2+\|u\|_q\|u\|_r \|\nabla^2 u\|_2 \right) 
            \\
            &\leq C \left(\|u\|_{W^{1,r}}\|u\|_{H^1}+\| u\|_{H^2}\|u\|_{W^{1,r}}\right)\|u\|_{H^1}
        \end{aligned}
    \end{equation}
    for all $r>2$. Combining \eqref{Q-estimation-uuu-1} and \eqref{Q-estimation-uuu-2} and using Young's inequality and the assumption $\|\alpha+\kappa\|_\infty\leq 1$ yields
    \begin{equation}
        \label{Q-estimation-uuu-3}
        \begin{aligned}
            \frac{2}{3\Pra}\bigg|\int_{\gamma^-\cup\gamma^+} (\alpha+\kappa) &u\cdot(u\cdot \nabla)u \ dS\bigg| 
            \\
            &\leq C\frac{1}{\Pra} \left(\|u\|_{W^{1,r}}\|u\|_{H^1}+\| u\|_{H^2}\|u\|_{W^{1,r}}\right)\|u\|_{H^1}
            \\ 
            &\leq \epsilon \|u\|_{H^2}^2 + C\left( C_\epsilon\left(\frac{1}{\Pra} \|u\|_{W^{1,r}}\right)^2+1\right) \|u\|_{H^1}^2
        \end{aligned}
    \end{equation}
    for all $\epsilon>0$ where $C_\epsilon>0$ depends on $|\Omega|$, $r$, $\|h'\|_{\infty}$ and $\epsilon$.
    \item
    In order to estimate the eleventh term on the right-hand side of \eqref{Q-estimation-1} notice that by Trace Theorem and Young's inequality
    \begin{align}
        \label{Q-estimation-un}
        2\Ra \left|\int_{\gamma^-\cup\gamma^+} (\alpha+\kappa) u_\tau n_1 \ dS\right| &\leq C \Ra \|\alpha+\kappa\|_\infty \|u\|_{H^1} \leq C \|\alpha+\kappa\|_\infty^2\Ra^2 +  \|u\|_{H^1}^2.
    \end{align}
\end{itemize}
In order to apply these estimates we first notice that by Lemma \ref{lemma_u_bounded_by_omega}
\begin{align*}
    \|u\|_{W^{1,r}}\leq C \left(\|\omega\|_r + (1+\|\kappa\|_\infty)^{2-\frac{2}{r}}\|u\|_2\right).
\end{align*}
The $L^p$-norm of the vorticity and energy are bounded by Lemma \ref{LemmaVorticityBound} and Lemma \ref{lemmaEnergyDecay} respectively, implying
\begin{align*}
    \|u\|_{W^{1,r}}\leq C \left(\|u_0\|_{W^{1,r}} +\underline{\alpha}^{-1}\Ra\right),
\end{align*}
where we exploited \eqref{proof-main-theorem-essinf-alpha-and-kappa-leq-1}.
Using this bound for the $W^{1,r}$ norm of $u$, the prefactors in the individual estimates are independent of time. Then taking the long time average of \eqref{Q-estimation-up}, \eqref{Q-estimation-uuu-3} and \eqref{Q-estimation-un} and plugging the bounds into \eqref{Q-estimation-1} yields
\begin{align*}
        \mathcal{Q}[u,\theta,\eta]
        &\geq M \Ra^2 + \frac{1}{2} \langle |\nabla\eta|^2\rangle + \frac{1}{2} \langle |\nabla \theta|^2\rangle + 2\langle \theta u\cdot\nabla \eta\rangle + \frac{b}{\Ra}\langle |\nabla u|^2\rangle + \frac{b}{\Ra}\langle (2\alpha+\kappa)u_\tau^2\rangle_{\gamma^-\cup\gamma^+} 
        \\
        &\qquad + a\langle |\nabla \omega|^2\rangle - a\left[2\epsilon+C\|\alpha+\kappa\|_\infty^2 \right]\left\langle \|u\|_{H^2}^2\right\rangle - C_\epsilon a \|\alpha+\kappa\|_{W^{1,\infty}}^2 \Ra^2
        \\
        &\qquad - aC\left[ C_\epsilon\left(\frac{1}{\Pra}\left(\|u_0\|_{W^{1,r}} + \underline{\alpha}^{-1}\Ra\right)\right)^2+\|\alpha+\kappa\|_{W^{1,\infty}}^2+1 \right]\left\langle\|u\|_{H^1}^2\right\rangle 
        \\
        &\qquad - a^2 \Ra^2 \langle |\omega|^2\rangle.
\end{align*}
Choosing $M= C_\epsilon a\|\alpha+\kappa\|_{W^{1,\infty}}^2$
\begin{equation}
    \label{Q-estimation-2}
    \begin{aligned}
        \mathcal{Q}[u,\theta,\eta]
        &\geq \frac{1}{2} \langle |\nabla \theta|^2\rangle + 2\langle \theta u\cdot\nabla \eta\rangle + \frac{b}{\Ra}\langle |\nabla u|^2\rangle + \frac{b}{\Ra}\langle (2\alpha+\kappa)u_\tau^2\rangle_{\gamma^-\cup\gamma^+} 
        \\
        &\qquad + a\langle |\nabla \omega|^2\rangle - a\left[2\epsilon+C\|\alpha+\kappa\|_\infty^2 \right]\left\langle \|u\|_{H^2}^2\right\rangle
        \\
        &\qquad - aC\left[ C_\epsilon\left(\frac{1}{\Pra}\left(\|u_0\|_{W^{1,r}} + \underline{\alpha}^{-1}\Ra\right)\right)^2+\|\alpha+\kappa\|_{W^{1,\infty}}^2+1 \right]\left\langle\|u\|_{H^1}^2\right\rangle 
        \\
        &\qquad - a^2 \Ra^2 \langle |\omega|^2\rangle.
    \end{aligned}
\end{equation}
In order to estimate the second term on the right-hand side of \eqref{Q-estimation-2} use Lemma \eqref{lemma-estimate-theta-u-gradeta} to get
\begin{align*}
        \mathcal{Q}[u,\theta,\eta]
        &\geq \frac{b}{\Ra}\langle |\nabla u|^2\rangle + \frac{b}{\Ra}\langle (2\alpha+\kappa)u_\tau^2\rangle_{\gamma^-\cup\gamma^+} 
        \\
        &\qquad + a\langle |\nabla \omega|^2\rangle - a\left[3\epsilon+C\|\alpha+\kappa\|_\infty^2 \right]\left\langle \|u\|_{H^2}^2\right\rangle
        \\
        &\qquad - aC\left[ C_\epsilon\left(\frac{1}{\Pra}\left(\|u_0\|_{W^{1,r}} + \underline{\alpha}^{-1}\Ra\right)\right)^2+\|\alpha+\kappa\|_{W^{1,\infty}}^2+1 \right]\left\langle\|u\|_{H^1}^2\right\rangle 
        \\
        &\qquad - a^2 \Ra^2 \langle |\omega|^2\rangle- C_{\epsilon}\delta^6 a^{-1}\langle |\nabla u|^2\rangle.
\end{align*}
Next according to Lemma \eqref{lemma-H1-bounded-by-grad-and-bdry-terms} the first two terms on the right-hand side can be estimated by the $H^1$ norm, i.e.
\begin{align}
    \label{main-theorem-lemma-H1-bounded-by-grad-bdry-terms-explicite-used-formula}
    \frac{3b}{8\Ra}\langle |\nabla u|^2\rangle + \frac{b}{2\Ra} \langle (2\alpha+\kappa) u_\tau^2 \rangle_{\gamma^-\cup\gamma^+} \geq \frac{\underline{\alpha}b}{8\Ra}   \langle \|u\|_{H^1}^2 \rangle.
\end{align}
Then
\begin{align*}
        \mathcal{Q}[u,\theta,\eta]
        &\geq a\langle |\nabla \omega|^2\rangle - a\left[3\epsilon+C\|\alpha+\kappa\|_\infty^2 \right]\left\langle \|u\|_{H^2}^2\right\rangle
        \\
        &\qquad + \left[\frac{\underline{\alpha} b}{8\Ra} - a C_\epsilon\left(\left(\frac{1}{\Pra}\left(\|u_0\|_{W^{1,r}} +\underline{\alpha}^{-1}\Ra\right)\right)^2+\|\alpha+\kappa\|_{W^{1,\infty}}^2+1\right) \right]\left\langle\|u\|_{H^1}^2\right\rangle
        \\
        &\qquad - a^2 \Ra^2 \langle |\omega|^2\rangle + \left(\frac{5b}{8\Ra}-C_\epsilon\delta^6 a^{-1}\right) \langle |\nabla u|^2\rangle + \frac{b}{2\Ra}\langle (2\alpha+\kappa)u_\tau^2\rangle_{\gamma^-\cup\gamma^+}
\end{align*}
and by Lemma \ref{lemma_u_bounded_by_omega} and the smallness conditions \eqref{proof-main-theorem-essinf-alpha-and-kappa-leq-1}
\begin{align*}
    \langle \|u\|_{H^2}^2 \rangle \leq C \langle |\nabla \omega|^2\rangle + C (1+\|\dot\kappa\|_\infty) \langle \|u\|_{H^1}^2 \rangle.
\end{align*}
For $\mathcal{Q}$ one gets
\begin{align*}
        \mathcal{Q}[u,\theta,\eta]&\geq  a\left[1-3\epsilon C_1-C_2\|\alpha+\kappa\|_\infty^2 \right]\langle |\nabla \omega|^2\rangle 
        \\
        &\qquad + \left[\frac{\underline{\alpha} b}{8\Ra} - a 3\epsilon - a C_\epsilon\left(\left(\frac{1}{\Pra}\left(\|u_0\|_{W^{1,r}} +\underline{\alpha}^{-1}\Ra\right)\right)^2+\|\dot\alpha\|_\infty^2+\|\dot\kappa\|_\infty^2+1\right) \right]\left\langle\|u\|_{H^1}^2\right\rangle
        \\
        &\qquad - a^2 \Ra^2 \langle |\omega|^2\rangle + \left(\frac{5b}{8\Ra}-C\delta^6 a^{-1}\right) \langle |\nabla u|^2\rangle + \frac{b}{2\Ra}\langle (2\alpha+\kappa)u_\tau^2\rangle_{\gamma^-\cup\gamma^+},
\end{align*}
where we used Young's inequality and that $\|\alpha\|_\infty \leq 2$ as $\|\kappa\|_\infty \leq 1$ by \eqref{proof-main-theorem-essinf-alpha-and-kappa-leq-1} and $\|\alpha+\kappa\|_\infty \leq 1$. Setting $\epsilon = \frac{1}{6 C_1}$ and using the smallness assumption $\|\alpha+\kappa\|_\infty^2\leq \bar C= \frac{1}{2C_2}$ the first bracket is positive and we are left with
\begin{equation}
    \label{Q-estimate-before-a0}
    \begin{aligned}
        \mathcal{Q}[u,\theta,\eta] &\geq \left[\frac{\underline{\alpha} b}{8\Ra} - a C\left(\left(\frac{1}{\Pra}\left(\|u_0\|_{W^{1,r}} +\underline{\alpha}^{-1}\Ra\right)\right)^2+\|\dot\alpha\|_\infty^2+\|\dot\kappa\|_\infty^2+1\right) \right]\left\langle\|u\|_{H^1}^2\right\rangle
        \\
        &\qquad - a^2 \Ra^2 \langle |\omega|^2\rangle + \left(\frac{5b}{8\Ra}-C\delta^6 a^{-1}\right) \langle |\nabla u|^2\rangle + \frac{b}{2\Ra}\langle (2\alpha+\kappa)u_\tau^2\rangle_{\gamma^-\cup\gamma^+}.        
    \end{aligned}
\end{equation}
Next we have to differentiate between the two conditions on $\kappa$.
\begin{itemize}
    \item Case $|\kappa|\leq \alpha$
    \vspace{10pt}\\
        In order to estimate the vorticity term notice that by Lemma \ref{lemma_u_bounded_by_omega} and the condition $|\kappa|\leq \alpha$ one has
        \begin{align*}
            \|\omega\|_2^2 &\leq \|\nabla u\|_2^2 + \int_{\gamma^-\cup\gamma^+} |\kappa| u_\tau^2 \leq \|\nabla u\|_2^2 + \int_{\gamma^-\cup\gamma^+} \alpha u_\tau^2 
            \\
            &\leq \|\nabla u\|_2^2 + \int_{\gamma^-\cup\gamma^+} (2\alpha+\kappa) u_\tau^2
        \end{align*}
        and taking the long time average \eqref{Q-estimate-before-a0} turns into
        \begin{align*}
            \mathcal{Q}[u,\theta,\eta] &\geq \left[\frac{\underline{\alpha} b}{8\Ra} - a C\left(\left(\frac{1}{\Pra}\left(\|u_0\|_{W^{1,r}} +\underline{\alpha}^{-1}\Ra\right)\right)^2+\|\dot\alpha\|_\infty^2+\|\dot\kappa\|_\infty^2+1\right) \right]\left\langle\|u\|_{H^1}^2\right\rangle
            \\
            &\qquad + \left[\frac{b}{2\Ra}- a^2\Ra^2\right]  \langle |\omega|^2\rangle + \left(\frac{b}{8\Ra}-C\delta^6 a^{-1}\right) \langle |\nabla u|^2\rangle.
        \end{align*}
        From the second squared bracket on the right-hand side it becomes clear that $a$ has to decay at least as fast as $\Ra^{-\frac{3}{2}}$ for $\mathcal{Q}$ to be non negative. Setting $a=a_0 \Ra^{-\frac{3}{2}}$
        \begin{align*}
            \mathcal{Q}[u,\theta,\eta] &\geq \left[\frac{\underline{\alpha} b}{8\Ra} - \frac{a_0 C}{\Ra^{\frac{3}{2}}} \left(\left(\frac{1}{\Pra}\left(\|u_0\|_{W^{1,r}} +\underline{\alpha}^{-1}\Ra\right)\right)^2+\|\dot\alpha\|_\infty^2+\|\dot\kappa\|_\infty^2+1\right) \right]\left\langle\|u\|_{H^1}^2\right\rangle
            \\
            &\qquad + \frac{1}{2\Ra}\left[b- 2a_0^2\right]  \langle |\omega|^2\rangle + \left(\frac{b}{8\Ra}-C\delta^6 a_0^{-1}\Ra^\frac{3}{2}\right) \langle |\nabla u|^2\rangle.
        \end{align*}
        The assumption $\Pra\geq \frac{1}{\underline{\alpha}^\frac{3}{2}}\Ra^\frac{3}{4}$ implies
        \begin{align*}
            \mathcal{Q}[u,\theta,\eta] &\geq \frac{\underline{\alpha}}{\Ra} \left[\frac{b}{8} - a_0 C \left(\|u_0\|_{W^{1,r}}^2+1+\underline{\alpha}^{-1}\Ra^{-\frac{1}{2}}\left(\|\dot\alpha\|_\infty^2+\|\dot\kappa\|_\infty^2+1\right)\right) \right]\left\langle\|u\|_{H^1}^2\right\rangle
            \\
            &\qquad + \frac{1}{2\Ra}\left[b- 2a_0^2\right]  \langle |\omega|^2\rangle + \left(\frac{b}{8\Ra}-C\delta^6 a_0^{-1}\Ra^\frac{3}{2}\right) \langle |\nabla u|^2\rangle
        \end{align*}
        and since $\Ra^{-\frac{1}{2}}<\underline{\alpha}$
        \begin{align*}
            \mathcal{Q}[u,\theta,\eta] &\geq \frac{\underline{\alpha}}{\Ra} \left[\frac{b}{8} - a_0 C \left(\|u_0\|_{W^{1,r}}^2+\|\dot\alpha\|_\infty^2+\|\dot\kappa\|_\infty^2+1\right) \right]\left\langle\|u\|_{H^1}^2\right\rangle
            \\
            &\qquad + \frac{1}{2\Ra}\left[b- 2a_0^2\right]  \langle |\omega|^2\rangle + \left(\frac{b}{8\Ra}-C\delta^6 a_0^{-1}\Ra^\frac{3}{2}\right) \langle |\nabla u|^2\rangle,
        \end{align*}
        where without loss of generality $C\geq 1$. In order for the two squared brackets to be non-negative we choose
        \begin{align*}
            a_0 = \frac{b}{8C\left(\|u_0\|_{W^{1,r}}^2+\|\dot\alpha\|_\infty^2+\|\dot\kappa\|_\infty^2+1\right)}
        \end{align*}
        and get
        \begin{align*}
            \mathcal{Q}[u,\theta,\eta]
            &\geq \left(\frac{b}{8\Ra} - C\delta^6 a_0^{-1}\Ra^\frac{3}{2}\right)\langle |\nabla u|^2\rangle.
        \end{align*}
        Letting $\delta$ solve $\frac{b}{8\Ra}= C \delta^6 a_0^{-1}\Ra^\frac{3}{2}$, i.e.
        \begin{align*}
            \delta = \left(\frac{a_0 b}{8C}\right)^\frac{1}{6}\Ra^{-\frac{5}{12}}
        \end{align*}
        $\mathcal{Q}$ is non-negative. Now we can come the estimating the Nusselt number. By \eqref{nu-identity-including-Q}
        \begin{align*}
            (1-b(1+\max h -\min h))\Nu + b \leq M \Ra^2 + 2\langle |\nabla \eta|^2\rangle.
        \end{align*}
        The gradient can be estimated by \eqref{nabla-eta-identity}, which yields
        \begin{align*}
            \langle |\nabla \eta|^2\rangle \leq C \delta^{-1}
        \end{align*}
        and plugging in $\delta$ and $M=C a\|\alpha+\kappa\|_{W^{1,\infty}}^2$ and choosing $b = \frac{1}{2(1+\max h - \min h)}$ we find
        \begin{align*}
            \Nu \leq C_\frac{1}{2} \|\alpha+\kappa\|_{W^{1,\infty}}^2 \Ra^\frac{1}{2} + C_\frac{5}{12} \Ra^\frac{5}{12}
        \end{align*}
        with $C_\frac{1}{2} = C(1+\|u_0\|_{W^{1,r}}^2)^{-1}$ and $C_\frac{5}{12} = C \left(\|u_0\|_{W^{1,r}}+\|\dot\alpha\|_\infty+\|\dot\kappa\|_\infty+1\right)^\frac{1}{3}$.
    \item Case $|\kappa|\leq 2\alpha + \frac{1}{4\sqrt{1+(h')^2}}\sqrt{\alpha}$\vspace{10pt}\\
        Using Lemma \ref{lemma_u_bounded_by_omega}, Trace Theorem and $\|\kappa\|_\infty \leq 1$ we can bound the vorticity term by
        \begin{align*}
            \|\omega\|_2^2 \leq C \|u\|_{H^1}^2
        \end{align*}
        and taking the long time average \eqref{Q-estimate-before-a0} turns into
        \begin{align*}
            \mathcal{Q}[u,\theta,\eta] &\geq \left[\frac{\underline{\alpha} b}{8\Ra} - a C\left(\left(\frac{1}{\Pra}\left(\|u_0\|_{W^{1,r}} +\underline{\alpha}^{-1}\Ra\right)\right)^2+\|\dot\alpha\|_\infty^2+\|\dot\kappa\|_\infty^2+1\right) \right]\left\langle\|u\|_{H^1}^2\right\rangle
            \\
            &\qquad - C a^2 \Ra^2 \left\langle\|u\|_{H^1}^2\right\rangle + \left(\frac{5b}{8\Ra}-C\delta^6 a^{-1}\right) \langle |\nabla u|^2\rangle + \frac{b}{2\Ra}\langle (2\alpha+\kappa)u_\tau^2\rangle_{\gamma^-\cup\gamma^+}.        
        \end{align*}
        Again applying \eqref{main-theorem-lemma-H1-bounded-by-grad-bdry-terms-explicite-used-formula} we find
        \begin{equation}
            \label{proof-main-theorem-before-new-second-cases}
            \begin{aligned}
                \mathcal{Q}[u,\theta,\eta] &\geq \left[\frac{\underline{\alpha} b}{8\Ra} - a C\left(\left(\frac{1}{\Pra}\left(\|u_0\|_{W^{1,r}} +\underline{\alpha}^{-1}\Ra\right)\right)^2+\|\dot\alpha\|_\infty^2+\|\dot\kappa\|_\infty^2+1\right) \right]\left\langle\|u\|_{H^1}^2\right\rangle
                \\
                &\qquad + \left[\frac{\underline{\alpha}b}{8\Ra}  - C a^2 \Ra^2\right] \left\langle\|u\|_{H^1}^2\right\rangle + \left(\frac{b}{4\Ra}-C\delta^6 a^{-1}\right) \langle |\nabla u|^2\rangle,
            \end{aligned}
        \end{equation}
        which because of the second squared bracket on the right-hand side imposes the condition on $a$ to decay at least as fast as $\Ra^{-\frac{3}{2}}$. We differentiate between two choices of $a$.
        \begin{itemize}
            \item[\small{$\blacktriangleright$}] 
                Setting $a=a_0\Ra^{-\frac{3}{2}}$ in \eqref{proof-main-theorem-before-new-second-cases} and estimating similar to before, using the assumptions $\Ra^{-1}\leq \underline{\alpha}$ and $\Pra\geq \frac{1}{\underline{\alpha}^\frac{3}{2}}\Ra^\frac{3}{4}$, we find
                \begin{align*}
                    \mathcal{Q}[u,\theta,\eta]
                    &\geq \frac{\underline{\alpha}}{\Ra} \left[\frac{b}{8} - a_0 C\left(\|u_0\|_{W^{1,r}}^2+\underline{\alpha}^{-\frac{1}{2}}\left(\|\dot\alpha\|_\infty^2+\|\dot\kappa\|_\infty^2+1\right)\right) \right]\left\langle\|u\|_{H^1}^2\right\rangle
                    \\
                    &\qquad  +\frac{1}{\Ra}\left(\frac{\underline{\alpha}b}{8} - C a_0^2\right) \langle \|u\|_{H^1}^2\rangle + \left(\frac{b}{4\Ra} - C\delta^6 a_0^{-1}\Ra^\frac{3}{2}\right)\langle |\nabla u|^2\rangle.
                \end{align*}
                Choosing
                \begin{equation*}
                    \begin{gathered}
                        a_0=\frac{\underline{\alpha}^\frac{1}{2} b}{8} \frac{1}{C\left(\|u_0\|_{W^{1,r}}^2+\|\dot\alpha\|_\infty^2+\|\dot\kappa\|_\infty^2+1\right)},\qquad
                        \delta = \left(\frac{a_0 b}{4C}\right)^\frac{1}{6}\Ra^{-\frac{5}{12}},
                        \\
                        b = \frac{1}{2(1+\max h - \min h)}
                    \end{gathered}
                \end{equation*}
                $\mathcal{Q}$ is non-negative and hence
                \begin{align*}
                    \Nu &\leq  C a_0 \|\alpha+\kappa\|_{W^{1,\infty}}^2 \Ra^\frac{1}{2} + C \delta^{-1}
                    \\
                    &\leq C_\frac{1}{2} \underline{\alpha}^\frac{1}{2}\|\alpha+\kappa\|_{W^{1,\infty}}^2 \Ra^\frac{1}{2} + C_\frac{5}{12} \underline{\alpha}^{-\frac{1}{12}} \Ra^\frac{5}{12}
                \end{align*}
                with $C_\frac{1}{2} = C(1+\|u_0\|_{W^{1,r}}^2)^{-1}$ and $C_\frac{5}{12} = C \left(\|u_0\|_{W^{1,r}}+\|\dot\alpha\|_\infty+\|\dot\kappa\|_\infty+1\right)^\frac{1}{3}$.
            \item[\small{$\blacktriangleright$}]
                Setting $a=a_0\Ra^{-\frac{11}{7}}$ in \eqref{proof-main-theorem-before-new-second-cases} and estimating similar to before, using $\Pra\geq \Ra^\frac{5}{7}$, we find
                \begin{align*}
                    \mathcal{Q}[u,\theta,\eta]
                    &\geq \frac{1}{\Ra} \left[\frac{\underline{\alpha} b}{8} - a_0 C \left( \Ra^{-\frac{9}{7}} \|u_0\|_{W^{1,r}}^2 +  \underline{\alpha}^{-2} + \Ra^{-\frac{4}{7}} \left(\|\dot\alpha\|_\infty^2 + \|\dot\kappa\|_\infty^2 + 1\right)\right)\right]\left\langle\|u\|_{H^1}^2\right\rangle
                    \\
                    &\qquad + \frac{1}{\Ra}\left(\frac{\underline{\alpha}b}{8} - C a_0^2 \Ra^{-\frac{1}{7}}\right) \langle \|u\|_{H^1}^2\rangle + \left(\frac{b}{4\Ra} - C\delta^6 a_0^{-1}\Ra^\frac{11}{7}\right)\langle |\nabla u|^2\rangle,
                \end{align*}
                which after choosing
                \begin{equation*}
                    \begin{gathered}
                    a_0 = \frac{\underline{\alpha}b}{8}\frac{1}{C(\|u_0\|_{W^{1,r}}^2+\underline{\alpha}^{-2}+ \|\dot\alpha\|_\infty^2 + \|\dot\kappa\|_\infty^2 + 1)},\qquad \delta = \left(\frac{a_0 b}{4C}\right)^\frac{1}{6} \Ra^{-\frac{3}{7}},
                    \\
                    b = \frac{1}{2(1+\max h - \min h)}
                    \end{gathered}
                \end{equation*}
                is non-negative, implying
                \begin{align*}
                    \Nu \leq C a \|\alpha+\kappa\|_{W^{1,\infty}}^2 \Ra^2 + C \delta^{-1}\leq  C_\frac{3}{7} \Ra^\frac{3}{7},
                \end{align*}
                where $C_\frac{3}{7}= C \left(\|\alpha+\kappa\|_{W^{1,\infty}}^2 + \underline{\alpha}^{-\frac{1}{2}} + \underline{\alpha}^{-\frac{1}{6}}(\|u_0\|_{W^{1,r}}+\|\dot\alpha\|_\infty+\|\dot\kappa\|_\infty+1)^\frac{1}{3}\right)$.
        \end{itemize}
\end{itemize}

\section{Notation}
\begin{itemize}
    \item
    Perpendicular direction:
    \begin{align*}
        a^\perp = \begin{pmatrix}-a_2\\a_1\end{pmatrix}
    \end{align*}
    \item
    Vorticity:
    \begin{align*}
        \omega = \nabla^\perp \cdot u.
    \end{align*}
    \item
    $u_\tau = \tau\cdot u$ is the scalar velocity along the boundary.
    \\
    \item
    Tensor product:
    \begin{align*}
        A \ \colon B = a_{ij}b_{ij}
    \end{align*}
    \item
    If not explicitly stated differently $C$ will denote a positive constant, which might depend on the size of the domain $|\Omega|=\Gamma$, $\|h'\|_{\infty}$ and potentially on the exponent of the Sobolev norm.
    \\
    \item
    $n_+$ is the normal vector pointing upwards.
    \\
    \item
    $n_-$ is the normal vector pointing downwards.
    \\
    \item
    $n$ is the general normal vector with direction pointing outwards its domain.
    \\
    \item
    $\tau=n^\perp$ is the tangential vector oriented in the direction of $n^\perp$. 
    \\
    \item
    $\lambda$ is the parameterization of the boundary by arc-length in the direction of $\tau$.
    \\
    \item
    Variable in the curved domain: $y=\begin{pmatrix}y_1,y_2\end{pmatrix}\in \Omega$
    \\
    \item
    Variable in the straightened domain: $x=\begin{pmatrix}x_1,x_2\end{pmatrix}\in [0,\Gamma]\times[0,1]$
    \\
    \item
    $S$ denotes the integration variable over one-dimensional curves.
    \\
    \item
    If the Lebesgue and Sobolev norms are taken over the whole domain of definition of the function we abbreviate like the following
    \begin{align*}
        \|\alpha+\kappa\|_\infty &= \|\alpha+\kappa\|_{L^\infty(\gamma^-\cup\gamma^+)}, & \|u\|_p &= \|u\|_{L^p(\Omega)}
    \end{align*}
    \item
    For single integrals over the whole domain of integration we skip the integration variable for easier readability, i.e.
    \begin{align*}
        \int_{\gamma^-} \kappa u_\tau^2 &= \int_{\gamma^-} \kappa u_\tau^2 \ dS, & \int_\Omega \omega^2 = \int_{\Omega} \omega^2 \ dy
    \end{align*}
    \item
    $L^p(\gamma^-,\gamma^+)$, depending on $y_1$, denotes the $L^p$-norm along a vertical line defined by
    \begin{align*}
        \|u\|_{L^p(\gamma^-,\gamma^+)}^p=\int_{h(y_1)}^{1+h(y_1)} |u(y_1,y_2)|^p\ dy_2.
    \end{align*}
\end{itemize}

\section{Appendix}
\subsection{Some technical observations}
\begin{itemize}
    \item The curvature on the boundary.\\
        As the bottom boundary can be parameterized by $(y_1,h(y_1))$ the tangential is parallel to $(1,h')$. Taking into consideration the symmetry of the domain, the outward pointing convention for $n$ and the definition of $\tau$, normalizing yields for the normal and tangent vectors
        \begin{align}
            \label{appendix-normal-vector-representation}
            n_\pm &= \pm \frac{1}{\sqrt{1+(h')^2}}\begin{pmatrix}-h'\\1\end{pmatrix},& \tau_\pm &= \mp \frac{1}{\sqrt{1+(h')^2}}\begin{pmatrix}1\\h'\end{pmatrix}.
        \end{align}
        In order to calculate the curvature we find by explicitly calculating
        \begin{align}
            \label{appendix-derivative-of-n}
            \frac{d}{dy_1} \tau_\pm = -\frac{h''}{1+(h')^2} n_\pm
        \end{align}
        and as the arc length parameterization in direction of $\tau_\pm$ is given by
        \begin{align*}
            \lambda(y_1) &= \int_{0}^{y_1} \sqrt{1+(h'(s))^2} \ ds \textnormal{ on }\gamma^-,
            & \lambda(y_1) &= \int_{y_1}^{\Gamma} \sqrt{1+(h'(s))^2} \ ds \textnormal{ on }\gamma^+
        \end{align*}
        one gets
        \begin{align*}
            \frac{d}{d y_1}\lambda(y_1) = \mp \sqrt{1+(h')^2} \textnormal{ on }\gamma^\pm
        \end{align*}
        and using \eqref{appendix-derivative-of-n}
        \begin{align*}
            \frac{d}{d\lambda} \tau_\pm = \mp \frac{1}{\sqrt{1+(h')^2}} \frac{d}{dy_1} \tau_\pm = \pm \frac{h''}{(1+(h')^2)^\frac{3}{2}} n_\pm,
        \end{align*}
        implying
        \begin{align}
            \label{appendix-kappa-representation}
            \kappa = \pm \frac{h''}{(1+(h')^2)^\frac{3}{2}} \textnormal{ on } \gamma^\pm.
        \end{align}
    \item Argument for \eqref{id-kappa}
        \begin{align}
            \label{appendix-proof-id-kappa}
            n\cdot(\tau\cdot \nabla) u
            &= n\cdot (\tau\cdot \nabla) (u_\tau \tau)
            = n\cdot \frac{d}{d\lambda} (u_\tau \tau)
            = u_\tau \kappa\ n\cdot n + n\cdot \tau \frac{d}{d\lambda}u_\tau =\kappa u_\tau.
        \end{align}
    \item Argument for \eqref{id-kappaUtau2-1}, \eqref{id-kappaUtau2-2} and \eqref{id-kappaUtau2-3}\\
    Using \eqref{appendix-proof-id-kappa}
        \begin{align}
            \label{appendix-proof-id-kappaUtau2}
            n\cdot(u\cdot \nabla) u
            &=u_{\tau} n\cdot (\tau\cdot \nabla) u = \kappa u_\tau^2.
        \end{align}

\end{itemize}
\subsection{Scaling of the curvature and friction coefficient with respect to the Rayleigh number}
\hypertarget{kappaScalingTarget}{\textbf{Scaling of the curvature with $\Ra$:}}
In what follows we want to compare two systems in which we vary the height and temperature gap. Notice that, when varying these parameters, the boundary-height $h$ and therefore the curvature $\kappa$ remains the same. In the corresponding nondimensionalized systems instead, as the Rayleigh number changes, the boundary-height function $\hat h$ (and therefore the curvature $\hat \kappa$) will too, as we are now going to show.
\\\\
In order to clarify this point, we want to write the original system
\begin{align}\label{1}
        \partial_t u + u\cdot \nabla u +\nabla \frac{p}{\rho_0} -\nu \Delta u &= -\bar{\alpha} (T-T_0) g\notag
        \\
        \partial_t (T-T_0) + u\cdot \nabla (T-T_0) - \varkappa \Delta (T-T_0) &= 0
        \\
        \nabla \cdot u &= 0\notag
\end{align}
Here the flow evolves in a rectangular domain of height $H$, the bottom boundary is held at temperature $T_0$ and the top boundary at temperature $T_1$. The other parameters are: the density $\rho_0$ at temperature $T_0$, the viscosity $\nu$, the thermal diffusivity $\varkappa$ and the thermal expansion coefficient $\bar{\alpha}$. Moreover we assume that the boundaries are not flat and call $h=h(x_1)$ the boundary-height function. 
In order to nondimensionalize the problem we define the following transformation:
\begin{align*}
        x &= H\hat x 
        \\
         t&= \frac{H^2}{\varkappa}\hat t
        \\
         u&=\frac{\varkappa}{H}\hat u
         \\
         T&=T_0+(T_1-T_0)\hat T
         \\
         p&=\rho_0\left(\frac{\varkappa}{H}\right)^2\hat p\,.
\end{align*}
In the new (nondimensionalized) $\hat{\cdot}$ variables the system is given by
\begin{align}\label{2}
        \frac{1}{\Pr}(\partial_t \hat u + \hat u\cdot \nabla \hat u) +\nabla \hat p -\Delta \hat u &= \Ra \hat T e_z\notag
        \\
        \partial_t \hat T + \hat u\cdot \nabla \hat T-  \Delta \hat T &= 0
        \\
        \nabla \cdot \hat u &= 0\notag\,,
\end{align}
where
\begin{equation}
    \label{Pr_def}
    \Pr=\frac{\nu}{\varkappa}    
\end{equation}
and 
\begin{equation}
    \label{rel}
    \Ra=\frac{\bar{\alpha} g \delta T H^3}{\nu\varkappa}\,.
\end{equation}
We observe that the nondimensionalized boundary-height function $\hat h$ is now rescaled:
\begin{equation}\label{height-f}
h=H\hat h\,.
\end{equation}
Equations \eqref{rel} and \eqref{height-f} yield two observations:
\begin{itemize}
\item Observation 1: Changing the height gap $H$ of the original system, while keeping the temperature gap $\delta T$ and boundary-height function $h$ fixed, results in a change in $\hat h$, while also changing $\Ra$.
\item Observation 2: Changing the temperature gap $\delta T$ of the original system, while keeping the height gap $H$ and boundary-height function $h$ fixed, also result in a change in $\Ra$, while not changing the nondimensionalized boundary-height function $\hat h$.
\end{itemize}

Now recall that by \eqref{kappa_h_relation} the curvature of the nondimensionalized system satisfies
$$\kappa\propto h''\,$$
and notice that by the rescaling defined above we have 
\begin{align*}
    \kappa=\frac{1}{H}\hat{\kappa}, \qquad \kappa'=\frac{1}{H^2}\hat{\kappa}',
\end{align*}
and that 
\begin{equation}\label{kappa}
    \|\hat{\kappa}\|_{W^{1,\infty}}\propto H^2+H.
\end{equation}
We claim that by carefully changing both, the height gap $H$ and also the temperature gap $\delta T$, we can achieve any polynomial scaling of $\|\hat \kappa\|_{W^{1,\infty}}$ with respect to $\Ra$, i.e.
\begin{equation}
    \label{claim}
    \|\hat{\kappa}\|_{W^{1,\infty}}\propto \Ra^{\rho}\,,
\end{equation}
for any exponent $\rho$. 
In order to show this, we compare two identical systems with height and temperature gap $(H_i,\delta T_i)$, $i=1,2$, where these parameter satisfy
\begin{align}
    \label{heightTemperatureRatio}
    \frac{H_{(2)}}{H_{(1)}}=\left(\frac{\delta T_{(2)}}{\delta T_{(1)}} \right)^{\frac{\rho}{2-3\rho}}.   
\end{align}
Then one has (at highest order as typical applications are in the regime $H\gg 1$)
\begin{align*}
    \frac{\|\hat \kappa_{(2)}\|_{W^{1,\infty}}}{\|\hat \kappa_{(1)}\|_{W^{1,\infty}}} 
    &\overset{\eqref{kappa}}{\approx} \left(\frac{H_{(2)}}{H_{(1)}}\right)^2 = \left(\frac{H_{(2)}}{H_{(1)}}\right)^{3\rho}\left(\frac{H_{(2)}}{H_{(1)}}\right)^{2-3\rho}
    \\    &\overset{\eqref{heightTemperatureRatio}}{=} \left(\frac{H_{(2)}}{H_{(1)}}\right)^{3\rho}\left(\frac{\delta T_{(2)}}{\delta T_{(1)}}\right)^{\rho} \overset{\eqref{rel}}{=} \left(\frac{\Ra_{(2)}}{\Ra_{(1)}}\right)^\rho.
\end{align*}
Note that the extreme cases $\rho = \frac{3}{2}$ and $\rho=0$ are covered by systems with the same temperature gap $\delta T$, respectively height $H$.

\textbf{Scaling of the roughness coefficient with $\Ra$:}
A similar argument as for the scaling of the curvature shows a possible scaling of $\alpha$ with respect to $\Ra$: The Navier-Slip boundary conditions can be derived as the small amplitude and high oscillation limit of the boundary roughness with no slip conditions (see \cite{miksisDavisSlipOverRoughAndCoatedSurfaces}). In this derivation the slip coefficient is given by the average of the height function of this roughness. Following the argument for the curvature scaling the slip coefficient in the nondimensionalized system scales as $\hat\alpha\sim \hat h$ and therefore a similar argument yields any scaling of the slip coefficient with respect to the Rayleigh number.

\subsection{Ideas for the proof with different bottom and top boundaries}
In this paper we made the assumption that the top and bottom boundaries are described by the same profile. This assumption is clearly not physical, but was made to not introduce further technical aspects to the proof.  
In order to treat different profiles at the top and bottom boundaries, some parts of the proof would need to be adjusted. Here, without giving details, we explain where the main changes happen. 

A key ingredient is that the Nusselt number identities hold in a similar fashion. In fact Proposition \ref{temp_review_prop_nusselt_identities} would read
\begin{align}
    \label{different_profiles_nu_representation}
    \Nu
    &= \langle |\nabla T|^2 \rangle
    = \langle (uT-\nabla T)\cdot n_+  \rangle_{\gamma^-(x_2)}
    \geq \frac{\langle (u_2-\partial_2)T\rangle}{\max h^+-\min h^-},
\end{align}
for any $0\leq x_2<\min h^+- \max h^-$, where $h^-$ and $h^+$ are the bottom and top boundary-height functions. In fact the first two identities are identical to \eqref{nusselt-gradT} and \eqref{nusselt-strip} as the shape of the top boundary does not influence the proofs, provided that $x_2$ is sufficiently small with respect to the height gap between the boundaries. 

The proof for the lower bound in \eqref{different_profiles_nu_representation} is the same as the proof for \eqref{nusselt-ineq} with the observation that the contribution arising from the top boundary, i.e. the term in $B$, can be neglected because of its sign in \eqref{nusselt_change_different_profiles}. Additionally note that for the proof of Theorem \ref{Lemma-Ra-One-Half-Bound} in this context it would be more convenient to localize the Nusselt number at the bottom boundary in order to match its definition and representation in \eqref{different_profiles_nu_representation}.

Lemma \ref{lemma-H1-bounded-by-grad-and-bdry-terms} would need to be altered significantly. This would result in worse conditions for the curvature in both our results, Theorem \ref{Lemma-Ra-One-Half-Bound} and \ref{main-theorem}. Notice that the background field, i.e. \eqref{def-eta} will now need to match the different top and bottom boundaries. Observe also that the cancellation in \eqref{different_boundaries_not_0_anymore} would not hold anymore, but one could obtain similar bounds by appropriately defining a vector field $\zeta$ that matches the normal at the boundaries. Then, as in \eqref{different_boundaries_not_0_anymore}, for every $y_1$ there exists an $\bar y_2$ such that $\partial_2 (\zeta \cdot u)=0$.
%
The elliptic regularity results in Lemma \ref{lemma_u_bounded_by_omega} will hold, albeit the introduction of technical nuisance.

\subsection*{Acknowledgements}
The authors thank Steffen Pottel for useful feedback and suggestions regarding the manuscript.
FB acknowledges the support by the Deutsche Forschungsgemeinschaft (DFG) within the Research Training Group GRK 2583 "Modeling, Simulation and Optimization of Fluid Dynamic Applications”. CN was partially supported by DFG-TRR181 and GRK-2583. 
\printbibliography

\end{document}